\documentclass[review]{elsarticle}

\pdfoutput=1

\usepackage{lineno,hyperref}
\modulolinenumbers[5]

\journal{Journal}

\graphicspath{{./}{./figs/}}

  \def\clap#1{\hbox to 0pt{\hss#1\hss}}

\usepackage{bm}
\providecommand{\mat}[1]{\bm{#1}}%
\renewcommand{\vec}[1]{\mathbf{#1}}

\providecommand{\mA}{\ensuremath{\mat{A}}}

\providecommand{\mC}{\ensuremath{\mat{C}}}
\providecommand{\mD}{\ensuremath{\mat{D}}}

\providecommand{\mH}{\ensuremath{\mat{H}}}

\providecommand{\mQ}{\ensuremath{\mat{Q}}}

\providecommand{\mU}{\ensuremath{\mat{U}}}
\providecommand{\mV}{\ensuremath{\mat{V}}}
\providecommand{\mW}{\ensuremath{\mat{W}}}

\providecommand{\mZ}{\ensuremath{\mat{Z}}}

\providecommand{\va}{\ensuremath{\vec{a}}}

\providecommand{\ve}{\ensuremath{\vec{e}}}

\providecommand{\vq}{\ensuremath{\vec{q}}}

\providecommand{\vu}{\ensuremath{\vec{u}}}
\providecommand{\vv}{\ensuremath{\vec{v}}}
\providecommand{\vw}{\ensuremath{\vec{w}}}
\providecommand{\vx}{\ensuremath{\vec{x}}}
\providecommand{\vy}{\ensuremath{\vec{y}}}
\providecommand{\vz}{\ensuremath{\vec{z}}}

\newcommand{\hLambda}{\hat{\Lambda}}

\newcommand{\hmU}{\hat{\mU}}

\newcommand{\sQ}{\mathcal{Q}}

\newcommand{\bpi}{\bm{\pi}}
\newcommand{\bgamma}{\bm{\gamma}}
\newcommand{\bdelta}{\bm{\delta}}

\newcommand{\bmat}[1]{\begin{bmatrix}#1\end{bmatrix}}

\newcommand{\unit}[1]{\text{#1}}
\usepackage{amsmath,amssymb,amsthm}
\usepackage{subfig}
\usepackage{algorithm}
\usepackage{multirow}

\usepackage{kbordermatrix}
\renewcommand{\kbldelim}{[}
\renewcommand{\kbrdelim}{]}

\newcommand{\Ren}{\text{Re}}
\newtheorem{lemma}{Lemma}
\newtheorem{theorem}{Theorem}

\begin{document}

\begin{frontmatter}

\title{Data-driven dimensional analysis: algorithms for unique and relevant dimensionless groups}





\author[CSM]{Paul G.~Constantine}
\address[CSM]{Department of Computer Science, University of Colorado, Boulder, CO} \ead{paul.constantine@colorado.edu}

\author[SU1]{Zachary del Rosario} 
\address[SU1]{Department of Aeronautics and Astronautics, Stanford University, Durand Building, 496 Lomita Mall, Stanford, CA 94305}

\author[SU2]{Gianluca Iaccarino}
\address[SU2]{Department of Mechanical Engineering and Institute for Computational Mathematical Engineering, Stanford University, Building 500, Stanford, CA 94305}

\begin{abstract}
Classical dimensional analysis has two limitations: (i) the computed dimensionless groups are not unique, and (ii) the analysis does not measure relative importance of the dimensionless groups. We propose two algorithms for estimating unique and relevant dimensionless groups assuming the experimenter can control the system's independent variables and evaluate the corresponding dependent variable; e.g., computer experiments provide such a setting. The first algorithm is based on a response surface constructed from a set of experiments. The second algorithm uses many experiments to estimate finite differences over a range of the independent variables. Both algorithms are \emph{semi-empirical} because they use experimental data to complement the dimensional analysis. We derive the algorithms by combining classical semi-empirical modeling with active subspaces, which---given a probability density on the independent variables---yield unique and relevant dimensionless groups. The connection between active subspaces and dimensional analysis also reveals that all empirical models are \emph{ridge functions}, which are functions that are constant along low-dimensional subspaces in its domain. We demonstrate the proposed algorithms on the well-studied example of viscous pipe flow---both turbulent and laminar cases. The results include a new set of two dimensionless groups for turbulent pipe flow that are ordered by relevance to the system; the precise notion of \emph{relevance} is closely tied to the derivative based global sensitivity metric from Sobol' and Kucherenko~\cite{Sobol2009}.
\end{abstract}


\end{frontmatter}

\section{Introduction}

\noindent
Dimensional analysis yields insights into a physical system through careful examination of the system's units. The principle underlying dimensional analysis is that the relationships between physical quantities do not change if the measurement units are changed. For example, the relationship between the speed at which an object falls toward the ground and the height from which it was dropped does not depend on whether the height was measured in feet or meters. Palmer states, ``The units themselves are a human construct, but the fundamental relationships in the physical world have nothing whatsoever to do with human beings''~\cite{Palmer2008}. The goal of dimensional analysis is to identify a set of \emph{unitless} variables---often called \emph{dimensionless groups}~\cite{Palmer2008}---that constitute the essential, scale-free (i.e., dimensionless) physical relationship; the Reynolds number is a well-known example of a dimensionless group that arises in several scale-free relationships in fluid dynamics. The lack of scale in the relationship implies that results from simpler small-scale experiments can reveal large-scale phenomena. The celebrated Buckingham Pi Theorem (see, e.g.,~\cite[Chapter 1.2]{Barenblatt1996},~\cite[Chapter 4]{calvetti2013}, or~\cite[Chapter 4]{Palmer2008})---the essential theoretical result in dimensional analysis---provides the precise number of dimensionless groups given the system's quantities' units expressed in a set of base units (e.g., the seven SI units~\cite{NIST2008}).

Dimensional analysis has several well-known limitations. The Buckingham Pi Theorem does not construct unique dimensionless groups---only a linear subspace of exponents that produces dimensionless groups; any basis for the subspace is equally valid. Nor can the Pi Theorem identify which of the dimensionless groups, given a particular choice of basis, are most important in the physical relationship---i.e., what are the relevant and driving parameters of the scale-free system. Moreover, dimensional analysis provides no information on the mathematical form of the scale-free relationship; it only identifies which parameters constitute the relationship. Despite these limitations, dimensional analysis has some prominent successes, such as Kolmogorov's theory of turbulence~\cite{Kolmogorov1991}.

A common approach to establishing the mathematical form of the scale-free physical relationship is to combine the results of dimensional analysis---i.e., identifying dimensionless groups---with experimental measurements of the physical system; see, e.g., the famous account of Taylor estimating a nuclear explosion's yield from a photograph of the blast~\cite{Taylor1950}. The measurements are transformed into evaluations of the dimensionless groups, which become \emph{data} to fit a response surface. This approach is often called \emph{semi-empirical modeling}~\cite[Chapter 4]{calvetti2013}, and it ensures that any data-driven model will satisfy the principle of \emph{dimensional homogeneity}~\cite[Chapter 1]{Barenblatt1996}, i.e., that the mathematical form only sums quantities with the same dimension. Semi-empirical modeling is appropriate in physical systems since physical measurements are assumed to obey physical laws, in contrast to purely empirical modeling of nonphysical systems (e.g., social sciences), where such laws may not be justified. One major advantage of semi-empirical modeling over empirical modeling (i.e., curve fitting from the original experimental measurements that have units) is that the dimensional analysis often yields fewer independent, unitless variables than the original measured quantities. Thus, the experimenter needs fewer potentially expensive experiments to fit the response surface with the dimensionless groups; this is a type of \emph{dimension reduction}.

At this point, we must be careful with the nomenclature. A \emph{unit} is a unit of measurement such as a meter or a foot. In the context of dimensional analysis, a unit is a type of \emph{dimension}; for example, both meter and foot are types of the dimension \emph{length}. A quantity is \emph{unitless} or \emph{dimensionless} (we use the two interchangeably) if it does not have units; by construction, a dimensionless group such as the Reynolds number is dimensionless. However, \emph{dimension} can also refer to the number of independent coordinates that define a vector space; for example, a function $f:\mathbb{R}^m\rightarrow\mathbb{R}$ maps a point in the $m$-dimensional vector space $\mathbb{R}^m$ to the one-dimensional vector space $\mathbb{R}$. In this context, we refer to the number of independent inputs to a function as the dimension of its input space. And \emph{dimension reduction} means replacing or approximating a function $f:\mathbb{R}^m\rightarrow\mathbb{R}$ by some other function $g:\mathbb{R}^n\rightarrow\mathbb{R}$, where $n<m$. Our notion of dimension reduction is distinct from seeking a low-dimensional manifold that describes a given data set of high-dimensional vectors---as in principal component analysis~\cite{Jolliffe2002} or nonlinear dimensionality reduction~\cite{Lee2007}.

We propose a data-driven (i.e., semi-empirical) approach to identify \emph{unique} and \emph{relevant} dimensionless groups in a physical system. We consider an idealized physical system with $m+1$ measured quantities with units, where we distinguish between $m$ independent variables and one dependent variable. We call this situation \emph{idealized} for two reasons. First, choosing independent and dependent variables for a physical experiment is challenging and time-consuming in practice. Although our approach cannot inform this choice for a particular experiment, the approach does provide a principled strategy to identify simplified relationships for any choice of independent and dependent variables. Second, we assume that (i) the experiment can be run for arbitrary values of the independent variables with some connected set (e.g., a  range for each independent variable) and (ii) the associated dependent variable (i.e., the result of the experiment) is sufficiently accurate. In practice, physical constraints restrict the experimenter from assessing any arbitrary point the space of independent variables---e.g., a pipe whose length differs from specifications by something less than the machining precision. Moreover, in many physical experiments, there is often significant noise---and the noise may vary with different choices of the independent variables. The idealizations allow us to construct and analyze algorithms for the data-driven dimensional analysis. The algorithms are most appropriate in the context of \emph{computer experiments}~\cite{Sacks1989,santner2003}, where each experiment executes a computer simulation instead of a physical experiment; our numerical examples in Section \ref{sec:example} employ a computer model of viscous pipe flow. In experimental practice, the experimenter should assess the degree to which a given experimental setup satisfies the idealizations assumed in our algorithms.

Given the idealized system, we first use the Buckingham Pi Theorem to construct a basis for the exponents of the possible dimensionless groups. We derive a corollary from the Pi Theorem that reveals the structure of any empirical model. In particular, every empirical model is a so-called \emph{ridge function}~\cite{pinkus2015}, which is a function of fewer than $m$ linear combinations of the independent variables; we review ridge functions and derive the corollary in Section \ref{sec:da}. The ridge function structure yields insights into the \emph{active subspaces} of the empirical model. Given a weight function on the independent variables, a function's active subspace is the span of a set of \emph{important} directions; perturbing the function's inputs along these directions changes the function's output more, on average, than perturbing the inputs orthogonally to the active subspace. In our context, the weight function on the independent variables quantifies the physical regime under consideration. We review active subspaces and show their connection to the empirical model's ridge function structure in Section \ref{sec:dr}.

The active subspace's notion of importance leads to a useful definition for \emph{relevance} among the dimensionless groups, and this definition is related to particular derivative-based global sensitivity metrics proposed by Sobol' and Kucherenko~\cite{Sobol2009}; we define \emph{relevance} and develop its connection to global sensitivity metrics in Section \ref{sec:derivations}. The active subspace-based definition of \emph{relevance} has a fortunate consequence: it constructs a unique\footnote{The derived dimensionless groups are unique in the sense that eigenvectors are unique---that is, up to a normalizing factor.} set of dimensionless groups---as long as the eigenvalues of the matrix whose eigenspaces define active subspaces are separated. In Section \ref{sec:algs} we provide two algorithms for (i) constructing the unique dimensionless groups and (ii) ranking them according to relevance. We demonstrate the approach on a simple example of viscous pipe flow in Section \ref{sec:example}. And we conclude in Section \ref{sec:conclusions} with summarizing commentary.

\section{All empirical models are ridge functions}
\label{sec:da}

\noindent
In 1969, Bridgman wrote, ``The principal use of dimensional analysis is to deduce from a study of the dimensions of the variables in any physical system certain necessary limitations on the form of any possible relationship between those variables''~\cite{Bridgman1969}. At the time, dimensional analysis was a mature set of tools, and it remains a staple of the science and engineering curriculum because of its ``great generality and mathematical simplicity''~\cite{Bridgman1969}. In this section, we make Bridgman's ``necessary limitations'' precise by connecting the Buckingham Pi Theorem to a particular low-dimensional structure found in ridge functions. Recent statistics literature has explored the importance of dimensional analysis for statistical analyses, e.g., design of experiments~\cite{Albrecth2013} and regression analysis~\cite{Shen2014}. These works implicitly exploit the low-dimensional structure in the physical relationships.

Today's data deluge motivates researchers across mathematics, statistics, and engineering to pursue exploitable low-dimensional descriptions of complex, high-dimensional systems. Computing advances empower certain structure-exploiting techniques to impact a wide array of important problems. Successes---e.g., compressed sensing in signal processing~\cite{Donoho2006,Candes2006}, neural networks in machine learning~\cite{neuralnets96}, and principal components in data analysis~\cite{Jolliffe2002}---abound. We review \emph{ridge functions}~\cite{pinkus2015}, which exhibit a particular type of low-dimensional structure, and we show how that structure manifests in empirical models. 

\begin{figure}[ht]
\centering
\subfloat[]{
\includegraphics[width=.43\textwidth]{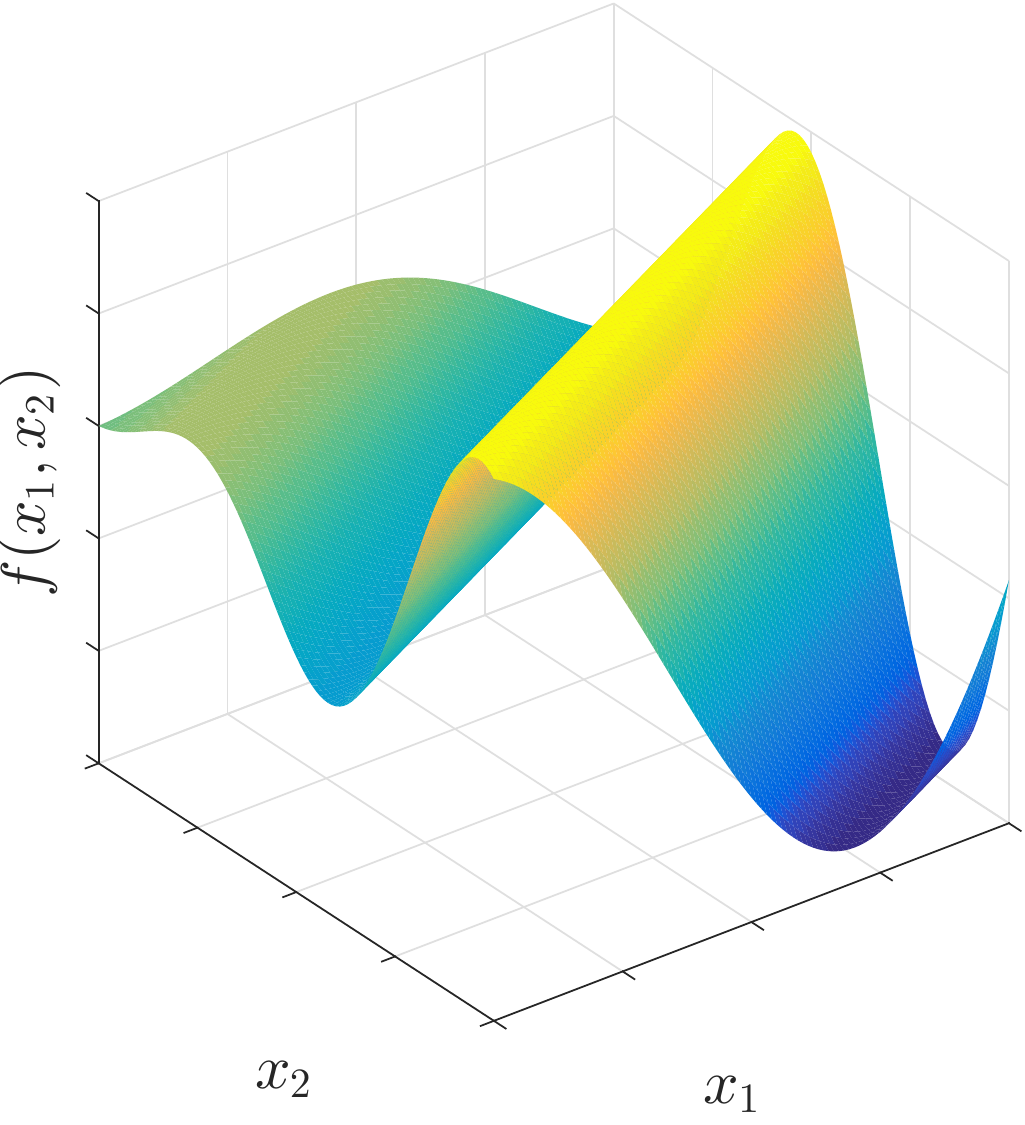}%
}
\subfloat[]{
\includegraphics[width=.43\textwidth]{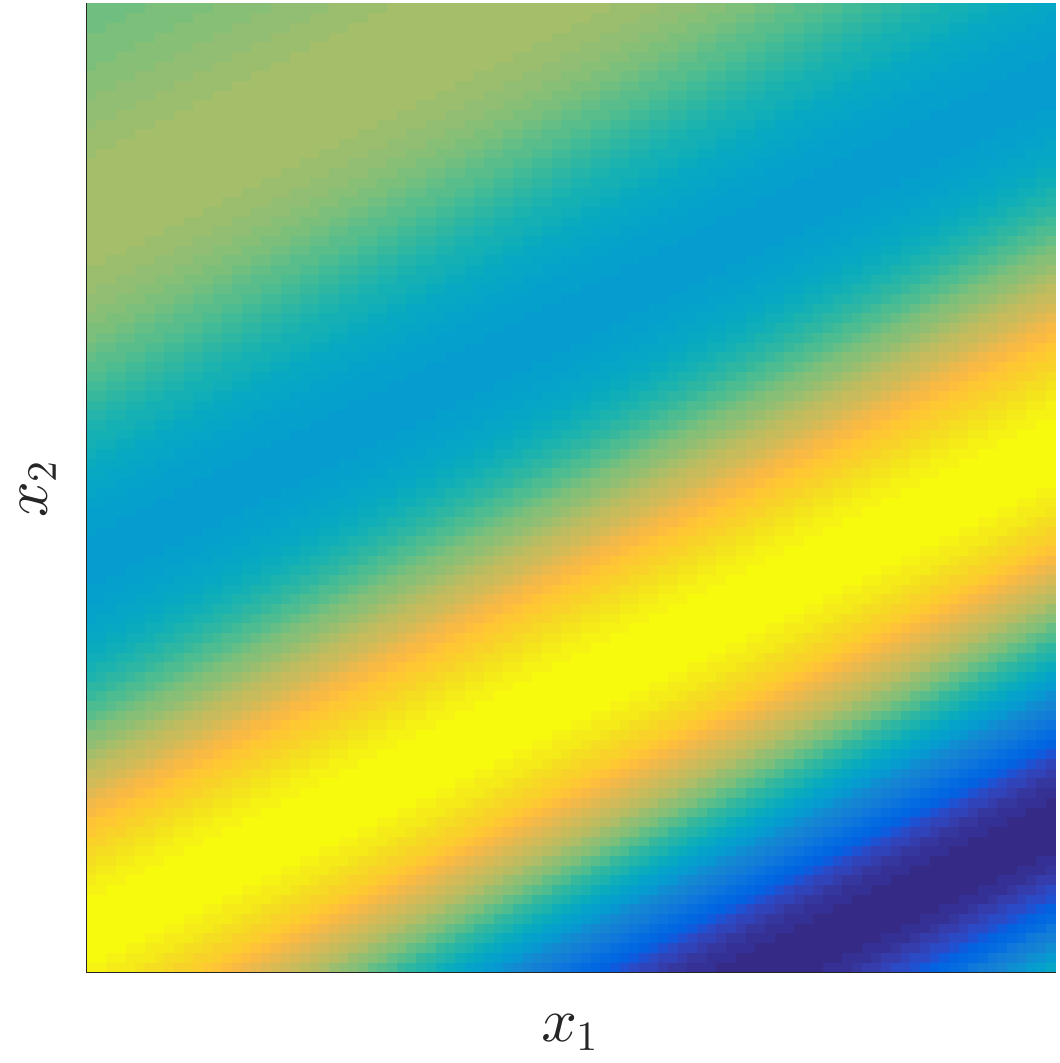}%
}
\caption{A ridge function [FIX ME]}
\label{fig:ridge}
\end{figure}

Let $\vx\in\mathbb{R}^m$ be a vector of independent, continuous variables; a ridge function $f:\mathbb{R}^m\rightarrow\mathbb{R}$ takes the form
\begin{equation}
\label{eq:ridgedef}
f(\vx) \;=\; g(\mA^T\vx),
\end{equation}
where $\mA\in\mathbb{R}^{m\times n}$ is a tall ($n<m$), full-rank matrix independent of $\vx$, and $g:\mathbb{R}^{n}\rightarrow\mathbb{R}$ is a scalar-valued function of $n$ variables. Although $f$ is nominally a function of $m$ variables, it is constant along all directions orthogonal to $\mA$'s columns. To see this, let $\vx\in\mathbb{R}^m$ and $\vy=\vx+\vu\in\mathbb{R}^m$ with $\vu$ orthogonal to $\mA$'s columns, i.e., $\mA^T\vu=0$. Then
\begin{equation}
\label{eq:constant}
f(\vy) 
\;=\; g(\mA^T(\vx + \vu)) 
\;=\; g(\mA^T\vx)
\;=\; f(\vx).
\end{equation}
Ridge functions appear in multivariate Fourier transforms, plane waves in partial differential equations, and statistical models such as projection pursuit regression and neural networks; see~\cite[Chapter 1]{pinkus2015} for a motivating introduction. Ridge functions are an object of study in approximation theory~\cite{Diaconis1984,Barron1993,pinkus2015}, and computational scientists have proposed methods for estimating their properties (e.g., the columns of $\mA$ and the form of $g$) from point evaluations $f(\vx)$~\cite{Fornasier2012,Cohen11,Tyagi2014}. However, scientists and engineers outside of mathematical sciences have paid less attention to ridge functions than other useful forms of low-dimensional structure. Many natural signals are sparse, and many real world data sets contain colinear factors. But whether ridge structures are pervasive in natural phenomena remains an open question. We answer this question affirmatively by showing that all empirical models are ridge functions as a consequence of dimensional analysis. To show this result, we first review classical dimensional analysis from a linear algebra perspective. 

\subsection{A linear algebra perspective on dimensional analysis}

\noindent
Several physics and engineering textbooks describe classical dimensional analysis. Barenblatt~\cite{Barenblatt1996} provides a thorough treatment in the context of scaling and self-similarity, while Ronin~\cite{sonin2001} is more concise. However, the presentation by Calvetti and Somersalo~\cite[Chapter 4]{calvetti15}, which ties dimensional analysis to linear algebra, is most appropriate for our purpose; what follows is similar to their treatment.  

We assume a chosen measurement system has $k$ base units---one for each of the $k$ dimensions. For example, if a mechanical system has $k=3$ dimensions of time $T$, length $L$, and mass $M$, then its base units may be seconds ($\unit{s}$), meters ($\unit{m}$), and kilograms ($\unit{kg}$), respectively. More generally, for a system in SI units~\cite{NIST2008}, $k$ is at most 7. All measured quantities in the system have units that are products of powers of the base units; for example, velocity has dimension length per time $(L/T)$ with units $\unit{m}\cdot\unit{s}^{-1}$. 

Define the \emph{dimension function} of a quantity $q$, denoted $[q]$, to be a function that returns the \emph{dimension} of $q$; e.g., if $q$ is velocity, then $[q]$ is $L/T$. If $q$ is dimensionless, then $[q]=1$. Define the \emph{dimension vector} of a quantity $q$, denoted $\vv(q)$, to be a function that returns the $k$ exponents of $[q]$ with respect to the dimensions of the $k$ base units; if $q$ is dimensionless, then $\vv(q)$ is a $k$-vector of zeros. For example, in a system with dimensions $T$, $L$, and $M$, if $q$ is velocity, then $[q]=L^1\cdot T^{-1}\cdot M^0$ and $\vv(q)=[1,-1,0]^T$; the order of the dimensions does not matter as long as they are consistent across quantities.

Barenblatt~\cite[Section 1.1.5]{Barenblatt1996} states that quantities $q_1,\dots,q_m$ ``have \emph{independent dimensions} if none of these quantities has a dimension function that can be represented in terms of a product of powers of the dimensions of the remaining quantities.'' This is equivalent to linear independence of the associated dimension vectors $\{\vv(q_1),\dots,\vv(q_m)\}$. 
We can express the exponents for derived units as a linear system of equations. Let $q_1,\dots,q_m$ contain quantities with units derived from $k$ base units, and $m\geq k$. Assume that a subset of $q_1,\dots,q_m$ of size $k$ has independent dimensions. Define the $k\times m$ matrix
\begin{equation}
\label{eq:D}
\mD \;=\; \bmat{\vv(q_1) &\cdots & \vv(q_m)}.
\end{equation}
By independence, $\mD$ has rank $k$. Let $p$ be a quantity with derived units $[p]$. Then $[p]$ can be written as products of powers of $[q_1],\dots,[q_m]$,
\begin{equation}
\label{eq:dimeq}
[p] \;=\; [q_1]^{w_1}\cdots[q_m]^{w_m}.
\end{equation}
The powers $w_1,\dots,w_m$ satisfy the linear system of equations
\begin{equation}
\label{eq:linsys}
\mD\,\vw = \vv(p), \qquad
\vw = \bmat{w_1\\ \vdots \\ w_m}. 
\end{equation}
Given the solution $\vw$ of \eqref{eq:linsys}, we can define a quantity $p'$ with the same units as $p$ (i.e., $[p]=[p']$) as
\begin{equation}
\label{eq:logtrans}
\begin{aligned}
p' &= q_1^{w_1}\cdots q_m^{w_m}\\
&= \exp\left(\log\left(
q_1^{w_1}\cdots q_m^{w_m}
\right)\right)\\
&= \exp\left(\sum_{i=1}^m
w_i\,\log(q_i)\right)\\
&= \exp\left(\vw^T\log(\vq)\right),
\end{aligned}
\end{equation}
where $\vq=[q_1,\dots,q_m]^T$, and the log of a vector returns the log of each component. There is some controversy over whether the logarithm of a physical quantity makes physical sense~\cite{molyneux91}. We sidestep this discussion by noting that $\exp(\vw^T\log(\vq))$ is merely a formal expression of products of powers of physical quantities. There is no need to interpret the units of the logarithm of a physical quantity. 

\subsection{Nondimensionalization}
\label{ssec:nondim}

\noindent
Assume we have an idealized physical system with $m+1$ quantities with units, $q$ and $\vq=[q_1,\dots,q_m]^T$, whose units are derived from a set of $k$ base units, and $m>k$. Without loss of generality, assume that $q$ is the quantity of interest (i.e., the dependent variable) with units $[q]$, and $[q]\not=1$---i.e., $q$ is not dimensionless. The remaining quantities $\vq$ are the independent variables. We assume that $\mD$, defined as in \eqref{eq:D}, has rank $k$, which is equivalent to assuming that there is a set of independent dimensions among $[q_1],\dots,[q_m]$. We construct a dimensionless independent variable $\pi$ as 
\begin{equation}
\label{eq:qoi}
\pi \;=\; \pi(q,\,\vq) \;=\; q\,\exp(-\vw^T\log(\vq)),
\end{equation}
where the exponents $\vw$ satisfy the linear system 
\begin{equation}
\label{eq:qoiw}
\mD\vw \;=\; \vv(q). 
\end{equation}
In practice, as the analyst is assembling the system's quantities, the $\mD$ she constructs may not be full rank, which may indicate some missing quantities; the solution $\vw$ to \eqref{eq:qoiw} is not guaranteed to exist when $\mD$ is not full rank. Our assumption that $\mD$ is full rank ignores this case. However, even under this assumption, the solution $\vw$ is not unique, since $\mD$ has a nontrivial null space (i.e., $m>k$ and $\text{rank}(\mD)=k$). For readibility, we speak of \emph{the} vector $\vw$; we mean \emph{any} vector that satisfies \eqref{eq:qoiw}.

Let $\mW=[\vw_1,\dots,\vw_n]\in\mathbb{R}^{m\times n}$ be a matrix whose columns contain a basis for the null space of $\mD$. In other words,
\begin{equation}
\label{eq:null0}
\mD\mW \;=\; \mathbf{0}_{k\times n},
\end{equation}
where $\mathbf{0}_{k\times n}$ is an $k$-by-$n$ matrix of zeros. Note that $n=m-k$, since $\text{rank}(\mD)=k$. Without loss of generality, we assume that $\mW$ has orthogonal columns, which is convenient for the linear algebra. The basis for the null space is not unique. To see this, define $\mV = \mW\mQ$, where $\mQ\in\mathbb{R}^{n\times n}$ is an orthogonal matrix, and note
\begin{equation}
\label{eq:null}
\mD\mV \;=\;\underbrace{\mD\mW}_{=\;\mathbf{0}}\mQ
\;=\; \mathbf{0}_{k\times n}.
\end{equation}
The basis' nonuniqueness represents a challenge in classical dimensional analysis. 
Calvetti and Somersalo~\cite[Chapter 4]{calvetti2013} offer a recipe for computing $\mW$ with rational elements via Gaussian elimination, which produces basis vectors with rational elements; this is consistent with the physically intuitive construction of many classical dimensionless groups, such as the Reynolds number. However, one must still choose the pivot columns in the Gaussian elimination. For readibility, we speak of \emph{the} basis $\mW$; we mean \emph{any} matrix that satisfies \eqref{eq:null0}.

Each column of $\mW$ represents a dimensionless group. Similar to \eqref{eq:qoi}, we can formally express the $n$ dimensionless groups, each denoted $\pi_i$, as
\begin{equation}
\label{eq:dimless}
\pi_i \;=\; \pi_i(\vq) \;=\; \exp(\vw_i^T\log(\vq)),\qquad i=1,\dots,n.
\end{equation}
The dimensionless groups depend on the choice of basis vectors $\{\vw_i\}$, so they are not unique.

The Buckingham Pi Theorem~\cite[Chapter 1.2]{Barenblatt1996} states that any physical relationship between $q$ and $\vq$ can be expressed as a relationship between the dimensionless dependent variable $\pi$ and the $n=m-k$ dimensionless groups $\pi_1,\dots,\pi_n$.  We seek a function $f:\mathbb{R}^{n}\rightarrow\mathbb{R}$ that mathematically represents the relationship between $\pi$ and $\pi_1,\dots,\pi_n$,
\begin{equation}
\label{eq:phi}
\pi \;=\; f(\pi_1,\dots,\pi_n).
\end{equation}
Expressing the relationship in dimensionless quantities has several advantages. First, there are often fewer dimensionless quantities than the original measured quantities, which allows one to construct $f$ empirically with many fewer experiments than one would need to model a mathematical relationship among the measured quantities; several classical examples showcase this advantage~\cite[Chapter 1]{Barenblatt1996}. Second, dimensionless quantities do not change if units are scaled, which allows one to devise small-scale experiments that reveal a scale-invariant relationship. Third, when all quantities are dimensionless, any mathematical relationship will satisfy \emph{dimensional homogeneity}, which is a physical requirement that models only sum quantities with the same dimension. 

\subsection{Derivation of ``All empirical models are ridge functions''}
\label{ssec:deriv}

\noindent
We exploit the nondimensionalized physical relationship \eqref{eq:phi} to show that the corresponding relationship among the original measured quantities can be modeled with a ridge function. To see this, we combine \eqref{eq:qoi},  \eqref{eq:dimless}, and \eqref{eq:phi} as follows:
\begin{equation}
\label{eq:rd0}
\begin{aligned}
q\,\exp(-\vw^T\log(\vq))
&= \pi \\
&= f(\pi_1,\dots,\pi_n)\\
&= f\left(
\exp(\vw_1^T\log(\vq)),\dots,\exp(\vw_n^T\log(\vq))
\right).
\end{aligned}
\end{equation}
We rewrite \eqref{eq:rd0} to emphasize the dependent variable $q$ as a function of the independent variables $\vq$, 
\begin{equation}
\label{eq:rd3}
q \;=\; \exp(\vw^T\log(\vq))\,\cdot\,f\left(
\exp(\vw_1^T\log(\vq)),\dots,\exp(\vw_n^T\log(\vq))
\right).
\end{equation}
We define $\vx=\log(\vq)$; that is, the vector $\vx\in\mathbb{R}^m$ contains the logarithms of the independent variables $\vq$. Then,
\begin{equation}
\label{eq:rd4}
\begin{aligned}
q &= \exp(\vw^T\vx)\,\cdot\,f\left(
\exp(\vw_1^T\vx),\dots,\exp(\vw_n^T\vx)
\right)\\
&= h(\vw^T\vx,\vw_1^T\vx,\dots,\vw_n^T\vx)\\
&= h(\mA^T\vx),
\end{aligned}
\end{equation}
where $h:\mathbb{R}^{n+1}\rightarrow\mathbb{R}$ and $\mA\in\mathbb{R}^{m\times (n+1)}$. The matrix $\mA$ contains the vectors computed in \eqref{eq:qoi} and \eqref{eq:dimless},
\begin{equation}
\label{eq:A}
\mA \;=\; 
\bmat{\vw & \mW} \;=\;
\bmat{\vw&\vw_1&\cdots &\vw_n}.
\end{equation}
The form \eqref{eq:rd4} is a ridge function in $\vx$, which justifies our thesis; compare to \eqref{eq:ridgedef}. We call the column space of $\mA$ the \emph{dimensional analysis subspace}. 

Several remarks are in order. First, \eqref{eq:rd4} reveals $h$'s dependence on its first coordinate. If one tries to fit $h(y_0,y_1,\dots,y_n)$ from measured data (i.e., as in semi-empirical modeling), then she should pursue a function of the form $h(y_0,y_1,\dots,y_n)=\exp(y_0)\,g(y_1,\dots,y_n)$, where $g:\mathbb{R}^n\rightarrow\mathbb{R}$. In other words, the first input of $h$ only scales another function of the remaining variables. 

Second, writing the physical relationship as $q=h(\mA^T\vx)$ as in \eqref{eq:rd4} uses a ridge function of the logs of the physical quantities. In dimensional analysis, some contend that the log of a physical quantity is not physically meaningful. However, there is no issue taking the log of numbers for semi-empirical modeling; measured data is often plotted on a log scale, which is equivalent. To construct $h$ from measured data, one must compute the logs of the measured numbers; such fitting is a computational exercise that ignores the quantities' units. 

To estimate $q$ given $\vq$ with the fitted semi-empirical model, we evaluate
\begin{equation}
\begin{aligned}
q &= h(\vw^T\log(\vq),\vw_1^T\log(\vq),\dots,\vw_n^T\log(\vq))\\
&= \exp(\vw^T\log(\vq))\,\cdot\,g\big(\log(\exp(\vw_1^T\log(\vq))),\dots, \log(\exp(\vw_n^T\log(\vq)))\big)\\
&= \exp(\vw^T\log(\vq))\,\cdot\,g\big(\log(\pi_1),\dots,\log(\pi_n)\big)
\end{aligned}
\end{equation}
In $g\left(\log(\pi_1),\dots,\log(\pi_n)\right)$, the logs act on nondimensional quantities, and $g$ returns a nondimensional quantity. By construction, the term $\exp(\vw^T\log(\vq))$ has the same units as $q$, so dimensional homogeneity is satisfied. Thus, the ridge function form of the semi-empirical model \eqref{eq:rd4} does not somehow violate dimensional homogeneity. 

Third, the columns of $\mA$ are linearly independent by construction. The first column is not in the null space of $\mD$ (see \eqref{eq:qoi}), and the remaining columns form a basis for the null space of $\mD$ (see \eqref{eq:null0}). So $\mA$ has full column rank. 

Finally---and most importantly---the ridge function structure in the semi-empirical model \eqref{eq:rd4} is independent of the particular response surface methodology used to construct $h$. In other words, the ridge function structure is fundamental to \emph{any} empirical model as a consequence of the dimensional analysis. The dependent variable $q$ is invariant to changes in the log-transformed independent variables $\vx=\log(\vq)$ that live in the null space of $\mA^T$; see \eqref{eq:constant}. Any attempt at empirical modeling that ignores this fundamental structure will miss valuable cost savings and exploitable insights. 

\subsection{Response surface methodologies}

\noindent
Since the ridge function structure in the generic semi-empirical model \eqref{eq:rd4} is independent of the choice of response surface, our analysis cannot help choose an appropriate response surface methodology for a particular application. There is an extensive literature on response surface methodolgies~\cite{Myers1995} and the requisite design of experiments~\cite{dean2017,santner2003}. Specific techniques include splines~\cite{Wahba1990} and radial basis functions~\cite{Wendland2004}. Jones provides a taxonomy of response surfaces~\cite{Jones2001}, and Shan and Wang survey available techniques for high-dimensional surfaces~\cite{Shan2009}; both reviews focus on the related context of optimization. Statistical tools for response surface modeling include regression surfaces~\cite{Hastie2009} and Gaussian processes~\cite{gpml2006} (also known as kriging surfaces~\cite{Stein1999}). The statistics-based tools include prediction variances derived from assumptions on noise in the data, and these prediction variances can be useful for designing experiments to improve the response surface's quality. In the related field of uncertainty quantification~\cite{Smith2013,Sullivan2015}, response surfaces based on polynomials go by the name \emph{polynomial chaos}~\cite{LeMaitre2010}.

\section{Active subspaces}
\label{sec:dr}

\noindent
The ridge function structure of the semi-empirical model \eqref{eq:rd4} implies that there are directions in the log-transformed space of the independent variables that do not change the dependent variable. This type of structure is related to the \emph{active subspaces} in a function of several variables. In this section, we review active subspaces and show their connection to the dimensional analysis subspace, i.e., the column space of $\mA$ from \eqref{eq:A}. 

The active subspaces~\cite{asbook,constantine2014active} of a given function are defined by sets of important directions in the function's domain. More precisely, let $\rho:\mathbb{R}^m\rightarrow\mathbb{R}_+$ be a bounded probability density function, and let $f:\mathbb{R}^m\rightarrow\mathbb{R}$ be a continuous, differentiable function with continuous and square-integrable (with respect to $\rho(\vx)$) partial derivatives. Define the $m\times m$ symmetric and positive semidefinite matrix $\mC$ as
\begin{equation}
\label{eq:Cmat}
\mC \;=\; \int \nabla f(\vx)\,\nabla f(\vx)^T\,\rho(\vx)\,d\vx,
\end{equation}
where $\nabla f(\vx)\in\mathbb{R}^m$ is the gradient of $f$. This matrix is the expected value of the rank-1-matrix-valued functional $\nabla f(\vx)\nabla f(\vx)^T$, where $\vx$ is a random vector distributed according to the joint density $\rho(\vx)$. Samarov~\cite{Samarov1993} studied a related matrix as one of several \emph{average derivative functionals} in the context of statistical regression, where $f$ is the link function in the regression model. In the same regression context, Hristache et al.~\cite{Hristache2001} uses a similar matrix to reduce the predictor dimension. See Cook~\cite{cook2009regression} for a comprehensive review of related \emph{sufficient dimension reduction} techniques for regression. Our context differs from regression, since there is no random noise in the dependent variable that is independent of $\vx$. 

The matrix $\mC$ admits a real eigenvalue decomposition $\mC=\mU\Lambda\mU^T$, where $\mU = \bmat{\vu_1 & \cdots & \vu_m}$ is the orthogonal matrix of eigenvectors, and $\Lambda$ is the diagonal matrix of non-negative eigenvalues denoted $\lambda_1,\dots,\lambda_m$ and ordered from largest to smallest. Writing the Rayleigh quotient form of the eigenvalue reveals (see Lemma 2.1 in~\cite{constantine2014active}),
\begin{equation}
\lambda_i \;=\; \int \big( \vu_i^T\nabla f(\vx)\big)^2\,\rho(\vx)\,d\vx.
\end{equation}
Assuming the $i$th eigenvalue $\lambda_i$ is unique, it measures the average, squared directional derivative of $f$ along the corresponding eigenvector $\vu_i$. In other words, the eigenvectors provide a set of orthogonal directions in the domain of $f$ that are ordered according to how perturbations in $\vx$ affect $f$. If $\lambda_i>\lambda_j$, then perturbing $\vx$ along $\vu_i$ changes $f$ more, on average, than perturbing $\vx$ along $\vu_j$. 

Assume that $\lambda_k>\lambda_{k+1}$ for some $k<m$ (i.e., $\lambda_k$ is \emph{strictly} greater than $\lambda_{k+1}$). Then the active subspace of dimension $k$ for the function $f$ is the span of the first $k$ eigenvectors; in short, active subspaces are eigenspaces of $\mC$ from \eqref{eq:Cmat}. Constantine developed computational procedures for (i) estimating the active subspace and (ii) exploiting it to enable calculations that are otherwise prohibitively expensive when the number $m$ of components in $\vx$ is large---e.g., approximation, optimization, and integration~\cite{asbook}. 

\subsection{Connections to ridge functions}

\noindent
A vector $\vu$ is in the null space of $\mC$ (i.e., $\mC\vu=0$) if and only if $f$ is constant along $\vu$; see Theorem 1 in~\cite{constantine2016near}. Thus, if $\mC$ is rank deficient, then $f(\vx)$ is a ridge function; see \eqref{eq:constant}. And if $f$ is a ridge function as in \eqref{eq:ridgedef}, then $f$'s active subspace is related to the $m\times n$ matrix $\mA$. First, observe that $\nabla f(\vx) = \mA\,\nabla g(\mA^T\vx)$, where $\nabla g$ is the gradient of $g$ with respect to its arguments. Then 
\begin{equation}
\label{eq:Cridge}
\mC \;=\; \mA\mH\mA^T,
\end{equation}
where
\begin{equation}
\label{eq:Hmat}
\mH\;=\;
\int \nabla g(\mA^T\vx)\,\nabla g(\mA^T\vx)^T\,\rho(\vx)\,d\vx.
\end{equation}
The symmetric positive semidefinite matrix $\mH$ has size $n\times n$. The form of $\mC$ in \eqref{eq:Cridge} implies two facts, which we state as lemmas.

\begin{lemma}
\label{lem:Crank}
Let $r = \text{rank}\,(\mH)$. Then for $\mC$ from \eqref{eq:Cridge}, $\text{rank}\,(\mC)\leq r$.
\end{lemma}

\begin{proof}
By~\cite[Section 0.4.5]{Horn1985}, 
\begin{equation}
\text{rank}(\mC) 
= \text{rank}(\mA\mH\mA^T)
\leq \min\{\,\text{rank}(\mA),\,\text{rank}(\mH)\,\}
= \min\{\,n,\,r\,\} = r,
\end{equation}
since $r\leq n$.
\end{proof}

\begin{lemma}
\label{lem:spaces}
The eigenspaces of $\mC$ from \eqref{eq:Cridge} are subspaces of $\mA$'s column space.
\end{lemma}

\begin{proof}
Let $r = \text{rank}(\mC)$, and let $\mU_k$ be a basis for the $k$-dimensional eigenspace of $\mC$ with $k\leq r$. Let $\vy\in\mathbb{R}^m$ be a vector in the $k$-dimensional eigenspace of $\mC$, i.e., $\vy=\mU_k\va_0$ for some $\va_0\in\mathbb{R}^k$. Let $\Lambda\in\mathbb{R}^{r\times r}$ be the diagonal matrix of non-zero eigenvalues of $\mC$. Then 
\begin{equation}
\vy = \mU_k \Lambda \Lambda^{-1} \va_0
= \mU_k \Lambda \va_1
= \mU_k \Lambda \mU_k^T \va_2
= \mC \va_2
= \mA\mH\mA^T \va_2
= \mA\va_3,
\end{equation}
where $\va_1 = \Lambda^{-1}\va_0$, $\va_2$ solves $\mU_k^T\va_2=\va_1$, and $\va_3=\mH\mA^T\va_2$.
\end{proof}


\subsection{A note on log transformations}
\label{ssec:logt}

\noindent
The derivation in Section \ref{ssec:deriv} shows that empirical models are ridge functions in the logarithms of the independent variables. Do Lemmas \ref{lem:Crank} and \ref{lem:spaces} still apply when the probability density $\rho$ is a function of the original variables but the ridge function is a function of the logarithms? The answer is yes, because the given density on the original variables induces a density on the logarithms---as long as (i) the range of the original variables is bounded and (ii) all values in the range are strictly positive; positivity can be ensured with an appropriate shift. Using the notation from Section \ref{sec:da}: let $\vq$ be the independent variables; let $\rho(\vq)$ be the probability density function on $\vq$, and let $\vx=\log(\vq)$. By the change-of-variables formula, the density on $\vx$, denoted $\sigma(\vx)$, given $\rho(\vx)$ is
\begin{equation}
\sigma(\vx) \;=\; \rho(\exp(\vx))\,\exp\left(\ve^T\vx\right),
\end{equation}
where $\ve$ is an $m$-vector of ones. Define the ridge function $f(\vq)=g(\mA^T\log(\vq))$. By a change of variables, the analog of \eqref{eq:Hmat} is
\begin{equation}
\int \nabla g(\mA^T\log(\vq))\,\nabla g(\mA^T\log(\vq))^T\,\rho(\vq)\,d\vq
\;=\;
\int \nabla g(\mA^T\vx)\,\nabla g(\mA^T\vx)^T\,\sigma(\vx)\,d\vx.
\end{equation}
In other words, Lemmas \ref{lem:Crank} and \ref{lem:spaces} still apply to active subspaces of a ridge function of the logarithms. The difference is in the density function. 

\subsection{Implications for semi-empirical modeling}

\noindent
Consider the idealized physical system defined by the dependent variable $q$ as a function of the log-transformed independent variables $\vx=\log(\vq)$. Assume that a chosen probability density function $\rho(\vx)$ quantifies the physical regime under consideration; for example, in Section \ref{sec:example} we study a pipe flow where chosen densities constrain the independent velocity, density, and viscosity such that the flow is either laminar or turbulent. Also, assume the function satisfies the smoothness conditions (i.e., continuity, differentiability) such that active subspaces are well defined.  

Since any empirical model is a ridge function (see \eqref{eq:rd4}), Lemma \ref{lem:Crank} implies that the number of dimensionless groups plus 1 is an upper bound on the dimension of this function's active subspaces. Moreover, Lemma \ref{lem:spaces} implies that any active subspace is a subspace of the dimensional analysis subspace (i.e., the column space of $\mA$ from \eqref{eq:A}). By computing a basis for the dimensional analysis subspace---e.g., with the Gaussian elimination-based approach~\cite[Chapter 4]{calvetti2013}---we obtain a subspace that contains all active subspaces without any gradient evaluations or numerical integration for estimating $\mC$ from \eqref{eq:Cmat}.


\section{Unique and relevant dimensionless groups}
\label{sec:derivations}

\noindent
Two well-known limitations of classical dimensional analysis are that (i) it does not produce unique dimensionless groups and (ii) it does not reveal the relative importance of the dimensionless groups in the physical relationship. We can address both of these concerns by studying active subspaces in the nondimenisonalized form of the physical relationship \eqref{eq:phi}. Our active subspace-based approach to the second limitation derives its notion of \emph{importance} from global sensitivity analysis.

\subsection{Sensitivity analysis}
\label{ssec:sensitivity}

\noindent
Global sensitivity analysis~\cite{saltelli2008global} seeks to measure the importance of each input variable for a function of several variables. Given a function $f(\vx)$, there are many metrics one may use to measure importance---each with its particular interpretation derived from its mathematical definition---including variance-based sensitivity indices (also known as Sobol' indices~\cite{SOBOL2001}), Morris' elementary effects~\cite{morris1991}, standardized regression coefficients~\cite{HELTON2006}, Shapley values~\cite{owen2014}, and derivative-based global sensitivity metrics from Sobol' and Kucherenko~\cite{Sobol2009}; see~\cite{saltelli2008global} for a review. Constantine and Diaz~\cite{diaz2015global} propose global sensitivity metrics derived from the eigenpairs of $\mC$ in \eqref{eq:Cmat} and connect them to common sensitivity metrics. 

We briefly review the derivative-based metrics from Sobol' and Kucherenko~\cite{Sobol2009}; the notion of \emph{relevance} in the dimensionless groups derived via active subspaces is closely related to their particular notion of \emph{importance} for global sensitivity. Consider a function $f:\mathbb{R}^m\rightarrow\mathbb{R}$ and a probability density function $\rho:\mathbb{R}\rightarrow\mathbb{R}_+$. The derivative-based global sensitivity metric $\nu_i$ for the $i$th independent variable $x_i$ is
\begin{equation}
\nu_i \;=\;
\int \left(
\frac{\partial}{\partial x_i}f(\vx)
\right)^2\,\rho(\vx)\,d\vx, \qquad i=1,\dots,m.
\end{equation}
The metric $\nu_i$ measures the average change in $f$ as $x_i$ is perturbed in the support of $\rho$. If $\nu_i>\nu_j$, then we expect that perturbing $x_i$ changes $f$ more than perturbing $x_j$, on average. This is the precise notion of \emph{importance} associated with the derivative-based global sensitivity metrics, and it is the notion we use to define \emph{relevance} of a dimensionless group.  Note that $\nu_i$ is also the $i$th diagonal of $\mC$ from \eqref{eq:Cmat}; see~\cite{diaz2015global} for a careful comparison between the active subspaces and the derivative-based global sensitivity metrics. 

\subsection{Transformed coordinates ordered by importance}
\label{ssec:trans}

\noindent
For a function $f$ with density $\rho$ as in Section \ref{sec:dr}, we can use the eigenvectors $\mU$ of $\mC$ from \eqref{eq:Cmat} to derive a new set of coordinates for $f$ that are naturally ordered according to the notion of importance associated with the derivative-based global sensitivity metrics. Define the coordinates $\vy=[y_1,\dots,y_m]^T$ as
\begin{equation}
\vy \;=\; \mU^T\vx. 
\end{equation}
Geometrically, the linear transformation $\mU^T$ rotates the domain of $f$. The density $\rho(\vx)$ from \eqref{eq:Cmat} induces a density $\sigma=\sigma(\vy)$ on $\vy$,
\begin{equation}
\sigma(\vy) \;=\; \rho(\mU\vy),
\end{equation}
since $\text{det}(\mU)=1$. Define a function of the rotated coordinates $g=g(\vy)$ as
\begin{equation}
g(\vy) \;=\; f(\mU\vy),
\end{equation}
and note its gradient with respect to $\vy$ is, by the chain rule,
\begin{equation}
\nabla_{\vy}g(\vy) \;=\; \nabla_{\vy} f(\mU\vy) \;=\; 
\mU^T\nabla_{\vx} f(\mU\vy),
\end{equation}
where the subscript on the gradient operator indicates which variables the derivatives are takeen with respect to. Each component of the gradient vector $\nabla_{\vy} g$ is
\begin{equation}
\frac{\partial}{\partial y_i}\,g(\vy) \;=\;
\vu_i^T\nabla_{\vx}f(\mU\vy),\qquad i=1,\dots,m.
\end{equation}
The derivative-based sensitivity metric $\nu_i$ for the $i$th component of $\vy$ with respect to $g$ is 
\begin{equation}
\label{eq:nu}
\nu_i \;=\; \int \left(\frac{\partial}{\partial y_i}\,g(\vy)\right)^2\,\sigma(\vy)\,d\vy,\qquad i=1,\dots,m.
\end{equation}
This number measures how much $g$ changes, on average, as $y_i$ is perturbed by a small amount. A quick calculation shows that $\nu_i$ is equal to the $i$th eigenvalue of $\mC$:
\begin{equation}
\begin{aligned}
\nu_i
&= \int \left(\frac{\partial}{\partial y_i}\,g(\vy)\right)^2\,\sigma(\vy)\,d\vy\\
&= \int \left( \vu_i^T\nabla_{\vx}f(\mU\vy) \right)^2\,\rho(\mU\vy)\,d\vy\\
&= \int \left( \vu_i^T\nabla_{\vx}f(\mU\mU^T\vx) \right)^2\,\rho(\mU\mU^T\vx)\,d\vx\\
&= \int \left( \vu_i^T\nabla_{\vx}f(\vx) \right)
\left( \nabla_{\vx}f(\vx)^T\vu_i \right)
\,\rho(\vx)\,d\vx\\
&= \vu_i^T\,
\left(\int \nabla_{\vx}f(\vx)\, \nabla_{\vx}f(\vx)^T\,\rho(\vx)\,d\vx\right)
\,\vu_i\\
&= \vu_i^T\,\mC\,\vu_i \;=\; \lambda_i.
\end{aligned}
\end{equation}
Since the eigenvalues are ordered in descending order---i.e., $\lambda_i\geq\lambda_j$ for $i<j$---the components of $\vy$ are ordered according to their importance as measured by the derivative-based sensitivity metrics $\nu_1,\dots,\nu_m$.

\subsection{Deriving unique and relevant dimensionless groups}
\label{ssec:unique}

\noindent
If we define \emph{relevance} as the notion of \emph{importance} from the derivative-based global sensitivity metrics---i.e., how very small changes to inputs affect the output, on average---then we can derive relevant dimensionless groups by computing active subspaces from the dimensionless form of the physical relationship \eqref{eq:phi}. To make the averaging precise, it is sufficient to choose a probability density function on the independent variables $\vq$. This density should quantify the physical regime of interest. For example, a density function with compact support could constrain the independent variables to ranges of values associated with desired physical phenomena---such as laminar versus turbulent flow or liquid versus gas phase. We make the following derivation simpler by assuming that each independent variable $q_i$ is constrained to the interval $[q_i^\ell,q_i^u]$, where $q_i^\ell>0$. The strict positivity of $q_i$ can always be satisfied by a shift. For example, if pressure fluctuation about a mean pressure is an independent variable, then shift the independent variable by the mean pressure so all possible inputs are strictly positive. Choose a bounded density function $\rho_i(q_i)$ with support on the interval $[q_i^\ell,q_i^u]$. In the absence of preference for particular subsets of $[q_i^\ell,q_i^u]$, a uniform density is sufficient (but not necessary); Jaynes expounds on the uniform density as an appropriate choice when there is no other preference or information~\cite[Chapter 12]{Jaynes2003}. Construct a joint density as a product,
\begin{equation}
\label{eq:rho}
\rho(\vq) = \prod_{i=1}^m \rho_i(q_i)
\qquad \mbox{with support} \qquad
\sQ = [q_1^\ell, q_1^u]\times\cdots\times [q_m^\ell, q_m^u].
\end{equation}
The product form is equivalent to assuming that the variables are independent in the probabilistic sense over $\sQ$. If one has information that the independent variables should not be probabilistically independent, then she may choose a joint density $\rho(\vq)$ that is not necessarily a product of univariate densities; a product density is not required for well defined active subspaces. However, the support should be shifted as needed to ensure that each component of $\vq$ remains positive.

The cost of uniqueness and relevance is that one must specify the density $\rho$ on the independent variables to define the averaging needed to compute active subspaces. Also, the unique and ranked-by-relevance dimensionless groups depend on the density $\rho$; if one is unsure of how to choose $\rho$, she may try several different $\rho$'s in independent experiments, which we can interpret as identifying unique and relevant dimensionless groups for different physical regimes. 

With $\rho$ chosen, we first create a new set of set of variables for the dimensionless relationship \eqref{eq:phi},
\begin{equation}
\label{eq:gg}
\begin{aligned}
\pi 
&= f(\pi_1,\dots,\pi_n)\\
&= f\big(
\exp(\log(\pi_1)),\dots,\exp(\log(\pi_n))
\big)\\
&= g\big(
\log(\pi_1),\dots,\log(\pi_n)
\big)\\
&= g(\gamma_1,\dots,\gamma_n),
\end{aligned}
\end{equation}
for some function $g:\mathbb{R}^n\rightarrow\mathbb{R}$, where $\gamma_i=\log(\pi_i)$. 

We derive a density function on $\bgamma=[\gamma_1,\dots,\gamma_n]^T$ from the density $\rho(\vq)$. Recall that $\vx=\log(\vq)$ from Section \ref{ssec:deriv}; for $\bpi=[\pi_1,\dots,\pi_n]^T$, by \eqref{eq:dimless},
\begin{equation}
\label{eq:gammamap}
\bgamma \;=\; \log(\bpi) \;=\; \mW^T\vx \;=\; \mW^T\log(\vq).
\end{equation}
The density function on $\bgamma$ is well defined by the map \eqref{eq:gammamap}; denote this density $\sigma(\bgamma)$. Assume that $g$ from \eqref{eq:gg} satisfies the conditions needed to get active subspaces (i.e., differentiable and square integrable derivatives with respect to $\sigma(\bgamma)$). Then active subspaces for the dimensionless relationship can be constructed from the eigendecomposition
\begin{equation}
\label{eq:Cg}
\int \nabla g(\bgamma)\, \nabla g(\bgamma)^T\,\sigma(\bgamma)\,d\bgamma
\;=\; \mU\Lambda\mU^T.
\end{equation}
As in Section \ref{ssec:trans}, we define a new set of coordinates $\bdelta=[\delta_1,\dots,\delta_n]^T$ by the orthogonal matrix $\mU$ of eigenvectors that are ordered by \emph{importance} as defined by the derivative-based sensitivity metrics. Let $\bdelta=\mU^T\bgamma$, and define the function $h=h(\bdelta)$ as
\begin{equation}
h(\bdelta) \;=\; g(\mU\bdelta).
\end{equation}
Assume the eigenvalues from \eqref{eq:Cg} are distinct, i.e., $\lambda_1>\cdots>\lambda_n$. Observe that
\begin{equation}
\label{eq:Z}
\bdelta \;=\; \mU^T\bgamma \;=\; \mU^T\mW^T\vx \;=\; \mZ^T\vx,
\end{equation}
where $\mZ=\mW\mU\in\mathbb{R}^{m\times n}$. Recall that $m$ is the number of the original independent variables, and $n$ is the number of dimensionless groups according to the Buckingham Pi Theorem; see Section \ref{ssec:nondim}. In other words, there is one column of $\mZ$ per dimensionless group, and the columns of $\mZ = \bmat{\vz_1 & \cdots & \vz_n}$ define the unique and relevant dimensionless groups. In particular, the unique dimensionless groups, denoted $\hat{\pi}_i$, are
\begin{equation}
\label{eq:udim}
\hat{\pi}_i \;=\; \hat{\pi}_i(\vq) \;=\; \exp\left(\vz_i^T\log(\vq)\right) \;=\; \exp(\delta_i), \qquad i=1,\dots,n,
\end{equation}
where $\delta_i$ is the $i$th component of $\bdelta$ from \eqref{eq:Z}. We emphasize that the sense in which the $\hat{\pi}_i$'s are \emph{unique} is the same sense of uniqueness for eigenvectors of a symmetric matrix---that is, unique up to a constant. More precisely, if $\vu_i$ is an eigenvector of the matrix in \eqref{eq:Cg}, then $\alpha\,\vu_i$, with $\alpha\in\mathbb{R}$, is also an eigenvector. This implies that, under this construction, the dimensionless group $\hat{\pi}_i^\alpha$ is indistinguishable from $\hat{\pi}_i$ from \eqref{eq:udim}. 

To see that the $\hat{\pi}_i$ are dimensionless, note
\begin{equation}
\mD\vz_i \;=\; \mD\mW\vu_i \;=\; \left[\mathbf{0}_{k\times n}\right]\,\vu_i \;=\; \mathbf{0}_{k\times 1}.
\end{equation}
In words, $\vz_i$ is in the null space of $\mD$, so it defines a dimensionless group. Compare the form of \eqref{eq:udim} to \eqref{eq:dimless}. The difference is in the vector of exponents; the vector $\vw_i$ from \eqref{eq:dimless} is not unique, whereas $\vz_i$ from \eqref{eq:udim} is unique---in the sense of eigenvectors---as seen in the following theorem.

\begin{theorem}
Assume the eigenvalues from \eqref{eq:Cg} are distinct, i.e., $\lambda_1>\cdots>\lambda_n$. Then the vectors $\{\vz_i\}$ from \eqref{eq:udim} are unique up to a constant.
\end{theorem}

\begin{proof}
Let $\mW\in\mathbb{R}^{m\times n}$ be a particular orthogonal basis for the null space of $\mD$ from \eqref{eq:null}. Let $\mQ$ be an orthgonal $n\times n$ matrix, and define $\mW'=\mW\mQ$. Define $\bgamma'=(\mW')^T\log(\vq)$ as in \eqref{eq:gammamap} and note
\begin{equation}
\bgamma' \;=\; \mQ^T\mW^T\log(\vq) \;=\; \mQ^T\bgamma.
\end{equation}
Let $g'(\bgamma')=g(\mQ\bgamma')$ for $g$ from \eqref{eq:gg}, and let $\sigma'(\bgamma') = \sigma(\mQ\bgamma')$ for $\sigma$ associated with the map \eqref{eq:gammamap}. Denote the gradient with respect to $\bgamma'$ as $\nabla'$. The active subspaces of $g'$ are defined by the matrix
\begin{equation}
\mC' \;=\; \int \nabla' g'(\bgamma')\,\nabla' g'(\bgamma')^T\,\sigma'(\bgamma')\,d\bgamma'.
\end{equation}
With a change of variables, since the determinant of $\mQ$ is 1,
\begin{equation}
\begin{aligned}
\mC' &= 
\mQ^T\,\left(
\int \nabla g'(\mQ^T\bgamma)\,\nabla g'(\mQ^T\bgamma)^T\,\sigma(\bgamma)\,d\bgamma
\right)\,\mQ\\
&= \mQ^T\,\left(
\int \nabla g(\bgamma)\,\nabla g(\bgamma)^T\,\sigma(\bgamma)\,d\bgamma
\right)\,\mQ\\
&= \mQ^T\,\left(\mU\Lambda\mU^T\right)\,\mQ\\
&= (\mU')\Lambda(\mU')^T,
\end{aligned}
\end{equation}
where $\mU'=\mQ^T\mU$. Define $\mZ'=\mW'\mU'$, and note
\begin{equation}
\mZ' = \mW'\mU' = \mW\mQ\mQ^T\mU = \mW\mU = \mZ,
\end{equation}
for $\mZ$ from \eqref{eq:Z}, as required. Note that this derivation is identical if one scales the eigenvectors $\mU$ by constants. 
\end{proof}

The proof requires that the eigenvalues of $\mC$ are separated so that the eigenvectors $\mU$ are unique up to a constant. If two or more eigenvalues are equal (i.e., if an eigenvalue has multiplicity greater than 1), then the associated columns of $\mZ$ are not unique; the associated eigenvectors only form a basis for the eigenspace. However, in this case, there is no preference for the ordering by relevance; the condition for unique dimensionless groups (i.e., distinct eigenvalues) is the same condition that implies some dimensionless groups are more relevant than others. In other words, nonuniqueness (in the sense of eigenvectors) occurs if and only if there is no prefence for relative importance, as indicated by the derivative-based sensitivity metrics \eqref{eq:nu}.

\section{Algorithms for estimating the unique dimensionless groups}
\label{sec:algs}

\noindent
We present two different approaches for numerically estimating the exponents of the unique and relevant dimensionless groups ($\{\vz_i\}$ from \eqref{eq:udim}). Recall the idealized physical system $m$ independent input variables $\vq=[q_1,\dots,q_m]^T$ and 1 dependent output variable $q$. An idealized \emph{experiment} chooses values for $\vq$ and determines the associated $q$. This determination could be the result of a physical experiment or a computational model. We assume that (i) the experimenter has precise control over the inputs and (ii) noise or error in the output is negligible. We recognize that such a clean set up is rare in practice; if the experimenter has information on errors or uncertainties in the system, the following algorithms should be modified to account for those errors. To account for generic noise or error, we would need to adopt or propose a mathematical model for generic noise; such models are hotly debated, and we consider an appropriate treatment of noise beyond the scope of the present work. 

For each algorithm, we assume that the experimenter first performs classical dimensional analysis to get (i) a particular vector $\vw$ that nondimensionalizes $q$ from \eqref{eq:qoi} and (ii) particular vectors $\{\vw_i\}$ for the dimensionless groups; see \eqref{eq:dimless}. One way to get $\{\vw_i\}$ is to compute the last $n$ singular vectors of $\mD$ from \eqref{eq:D}~\cite[Chapter 2.4]{gvl2013}. 

\subsubsection{Response surface-based algorithm}

\noindent 
The first algorithm is based on fitting a response surface---as in standard semi-empirical modeling---and estimating the eigenvectors that define the active subspaces from the surface's gradient. This idea has precedent. Yang et al.~\cite{Yang2016} build a global polynomial response (i.e., a \emph{polynomial chaos} surrogate) and compute gradients from this approximation. In the context of dimension reduction for regression, Fukumizu and Leng~\cite{fukumizu2014} build a kernel-based approximation (i.e., a radial basis approximation) and compute gradients. We do not specify the form of the response surface in the algorithm.

Assume the experimenter has performed an appropriate experimental design that produces input/output pairs $(\vq^{(j)},q^{(j)})$ with $j=1,\dots,N$, where the inputs $\{\vq^{(j)}\}$ are consistent with the chosen density $\rho$ from \eqref{eq:rho}. Algorithm \ref{alg:rs} starts with these pairs.  

\begin{algorithm}
Given a vector $\vw$ from \eqref{eq:qoi}, vectors $\vw_1,\dots,\vw_n$ from \eqref{eq:dimless}, and pairs $(\vq^{(1)},q^{(1)}), \dots, (\vq^{(N)},q^{(N)})$ from a design of experiments consistent with $\rho$ from \eqref{eq:rho}.
\begin{enumerate}
\itemsep0em
\item Compute evaluations of the dimensionless independent variable,
\begin{equation}
\pi^{(j)} \;=\; q^{(j)}\,\exp\left(-\vw^T\log(\vq^{(j)})\right).
\end{equation}
\item Compute evaluations of the logs of the dimensionless groups,
\begin{equation}
\gamma_i^{(j)} \;=\; \vw_i^T\log(\vq^{(j)}), \qquad i=1,\dots,n,
\end{equation}
and $\bgamma^{(j)}=[\gamma_1^{(j)},\dots,\gamma_n^{(j)}]^T$.
\item Fit a response surface $\hat{g}$ with the pairs $\{(\pi^{(j)},\bgamma^{(j)})\}$ such that
\begin{equation}
\label{eq:ghat}
\pi^{(j)} \;\approx\; \hat{g}(\bgamma^{(j)}),\qquad j=1,\dots,N.
\end{equation}
\item Use the response surface gradient $\nabla\hat{g}$ to define active subspaces:
\begin{equation}
\label{eq:gas}
\int \nabla \hat{g}(\bgamma)\, \nabla \hat{g}(\bgamma)^T\,\sigma(\bgamma)\,d\bgamma
\;\approx\; \hmU\hLambda\hmU^T.
\end{equation}
Estimate these integrals with any appropriate numerical integration rule; the approximation sign in \eqref{eq:gas} signifies the numerical approximation.
\item Compute the weight vectors $\{\hat{\vz}_i\}$ that define the unique and relevant dimensionless groups (see \eqref{eq:Z}),
\begin{equation}
\label{eq:zhatrs}
\hat{\vz}_i \;=\; \mW\hat{\vu}_i, \qquad i=1,\dots,n.
\end{equation}
\end{enumerate}
\caption{Response surface-based algorithm}
\label{alg:rs}
\end{algorithm}

The computed vectors $\{\hat{\vz}_i\}$ from \eqref{eq:zhatrs} in Algorithm \ref{alg:rs} have two essential numerical errors. First, the response surface $\hat{g}$ from \eqref{eq:ghat} may be an inaccurate approximation of the true $g$ from \eqref{eq:gg}---depending on the chosen design of experiments and the smoothness of $g$. Second, the numerical approximation of the integrals in \eqref{eq:gas} may be inaccurate---depending on the smoothness of the response surface gradients and the number $n$ of components in $\bgamma$ (i.e., the dimension of the integration problem). Constantine and Gleich analyzed a Monte Carlo method for estimating active subspaces~\cite{constantine2015computing}. In theory, both of these errors can be controlled with more experiments (i.e., larger $N$) and more accurate numerical integration.

The advantage of Algorithm \ref{alg:rs} is that it can be performed using an existing set of experiments without additional experimental data---assuming the existing experiments are sufficient to build the response surface $\hat{g}$ in \eqref{eq:ghat}. 

\subsubsection{Finite difference gradients}

\noindent 
The second algorithm is based on finite difference approximations of the gradient of $g$ from \eqref{eq:gg}. In contrast to Algorithm \ref{alg:rs}, the finite difference-based approach requires additional experiments to estimate gradients at an initial design of experiments. Assume the experimenter has performed an experimental design that produces input/output pairs $(\vq^{(j)},q^{(j)})$ with $j=1,\dots,N$, where the inputs $\{\vq^{(j)}\}$ are appropriate for numerical integration with respect to the chosen density $\rho$ from \eqref{eq:rho}. And let $\omega^{(1)},\dots,\omega^{(N)}$ be the weights of the associated numerical integration scheme. For a Monte Carlo scheme~\cite{constantine2015computing}, $\omega^{(j)} = 1/N$ and the $\vq^{(j)}$'s are drawn independently according to $\rho(\vq)$. 

\begin{algorithm}
Given a vector $\vw$ from \eqref{eq:qoi}, vectors $\vw_1,\dots,\vw_n$ from \eqref{eq:dimless}, an integration rule with point/weight pairs $(\vq^{(1)},\omega^{(1)}), \dots, (\vq^{(N)},\omega^{(N)})$ for approximating integrals with respect to $\rho$ from \eqref{eq:rho}, and associated experimental results $q^{(1)},\dots,q^{(N)}$.
\begin{enumerate}
\itemsep0em
\item Compute evaluations of the dimensionless independent variable,
\begin{equation}
\pi^{(j)} \;=\; q^{(j)}\,\exp\left(-\vw^T\log(\vq^{(j)})\right).
\end{equation}
\item Compute evaluations of the logs of the dimensionless groups,
\begin{equation}
\gamma_i^{(j)} \;=\; \vw_i^T\log(\vq^{(j)}), \qquad i=1,\dots,n,
\end{equation}
and $\bgamma^{(j)}=[\gamma_1^{(j)},\dots,\gamma_n^{(j)}]^T$.
\item For each $\bgamma^{(j)}$, for $k=1,\dots,n$, find $\vq^{(j,k)}$ that satisfy
\begin{equation}
\label{eq:fdpts}
\mW^T\log(\vq^{(j,k)}) = \bgamma^{(j)} + h\,\ve_k,\qquad j=1,\dots,N,
\end{equation}
where $\ve_k$ is a vector of zeros except for a 1 in the $k$th element and $h$ is the finite difference step size. 
\item For each $\vq^{(j,k)}$, run an experiment to get the corresponding $q^{(j,k)}$. 
\item Compute the corresponding dimensionless dependent variables
\begin{equation}
\pi^{(j,k)} \;=\; q^{(j,k)}\,\exp\left(-\vw^T\log(\vq^{(j,k)})\right).
\end{equation}
\item Compute finite differences
\begin{equation}
\frac{\partial}{\partial \gamma_k} g(\bgamma^{(j)})
\;\approx\; \frac{1}{h}(\pi^{(j,k)} - \pi^{(j)}),
\end{equation}
and let $\nabla_h g(\bgamma^{(j)})\in\mathbb{R}^n$ be the collection of finite differences.
\item Estimate the active subspaces as
\begin{equation}
\label{eq:evecsfd}
\sum_{j=1}^N \omega^{(j)}\,
\nabla_h g(\bgamma^{(j)})\, \nabla_h g(\bgamma^{(j)})^T 
\;=\; \hmU\hLambda\hmU^T. 
\end{equation}
\item Compute the weight vectors $\{\hat{\vz}_i\}$ that define the unique and relevant dimensionless groups (see \eqref{eq:Z}),
\begin{equation}
\label{eq:zhatfd}
\hat{\vz}_i \;=\; \mW\hat{\vu}_i, \qquad i=1,\dots,n.
\end{equation}
\end{enumerate}
\caption{Finite difference-based algorithm}
\label{alg:fd}
\end{algorithm}

Several comments regarding Algorithm \ref{alg:fd} are in order. The idea is to estimate the gradients of $g$ from \eqref{eq:gg} without constructing a global response surface. There are two potential numerical sources of errors: (i) the errors in the numerical integration and (ii) the errors in the finite difference approximations. Both of these errors are controllable using more integration points or smaller finite difference step size, respectively, subject to all the usual caveats associated with numerical approximations of integrals and derivatives. Algorithm \ref{alg:fd} uses a first-order finite difference scheme; of course, the algorithm may be adapted to use higher-order schemes~\cite{fornberg1988} at a cost of additional experiments. The points in the space of independent variables computed as in \eqref{eq:fdpts} are not unique, because the linear system is underdetermined. However, the differences in solutions do not affect the finite difference approximations of the dimensionless physical relationship's derivatives.  

A disadvantage of Algorithm \ref{alg:fd} is that it requires $Nn$ more experiments than the initial experimental design with $N$ points. One idea to address this issue is to employ the compressed sensing-based approach of Constantine et al.~\cite{constantine2015CAMSAP}, which uses fewer than $n$ random directional derivatives per quadrature point. 

\section{Example: viscous pipe flow}
\label{sec:example}

\noindent
We demonstrate the two main results of the paper---(i) that every semi-empirical model is a ridge function and (ii) the algorithms for computing unique and relevant dimensionless groups---using the classical example of viscous flow through a pipe; see~\cite[Chapter 6]{White2011} and~\cite{Moody1944,Colebrook1937}. Palmer also studies this problem in~\cite[Chapter 5.2]{Palmer2008}. The system's three base units ($k=3$) are kilograms (kg), meters (m), and seconds (s). The physical quantities include the fluid's bulk velocity $V$, density $\rho$, and viscosity $\mu$; the pipe's diameter $D$ and characteristic wall roughness $\varepsilon$; and the pressure loss $dp/dx$ between pipe ends. Table \ref{tab:pipedim} summarizes the quantities and their units. We treat $dp/dx$ as the quantity of interest; in other words, we imagine the scenario where the pipe engineer can control fluid properties, pipe diameter, and flow velocity (i.e., the independent variables) and measure the corresponding pressure loss (i.e., the dependent variable). The codes for reproducing the following experiments are available at \url{https://bitbucket.org/paulcon/pipe-code}.

\begin{figure}[ht]
\centering
\includegraphics[width=.71\textwidth]{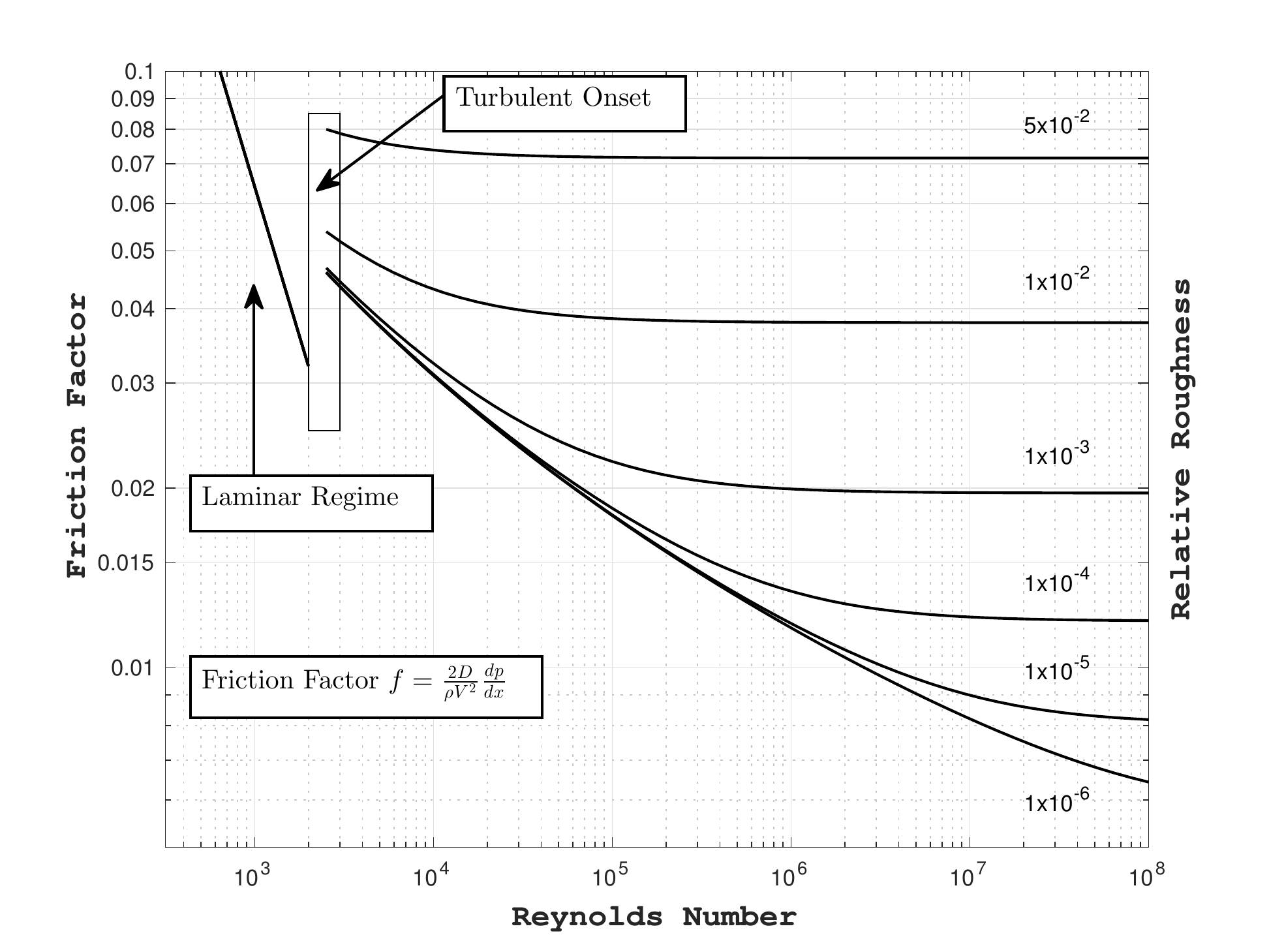}
\caption{The Moody Diagram plots the friction factor (dimensionless pressure loss) against the Reynolds number and relative roughness. Transition from laminar flow (governed by the Poiseuille relation) to turbulent flow (modeled by the Colebrook equation) occurs at a critical Reynolds number $\Ren_c \approx 3000$.}
\label{fig:moody}
\end{figure}

The Moody Diagram implicitly relates the physical quantities; Figure \ref{fig:moody}  plots the Fanning friction factor $\lambda$ defined by
\begin{equation}
\lambda \;=\; \frac{dp/dx\,D}{\frac{1}{2}\rho V^2},
\end{equation}
against the Reynolds number $\frac{\rho V D}{\mu}$ and relative roughness $\frac{\varepsilon}{D}$~\cite{Moody1944}. Below a critical Reynolds number around $\Ren_c=3\times10^3$, the friction factor satisfies the Poiseuille relation~\cite[Chapter 6]{White2011},
\begin{equation}
\label{eq:poiseuille}
\lambda \;=\; \frac{64}{\Ren}, 
\end{equation}
For $\Ren>\Ren_c$, the Colebrook equation~\cite{Colebrook1937} implicitly defines the relationship between friction factor and the other quantities,
\begin{equation}
\label{eq:colebrook}
\frac{1}{\sqrt{\lambda}} \;=\; -2.0\log_{10}\left(\frac{1}{3.7}\frac{\varepsilon}{D} + \frac{2.51}{\Ren\,\sqrt{\lambda}}\right), 
\end{equation}
The Colebrook equation is valid through transition to full turbulence. 

\begin{table}[ht]
\centering
\caption{Physical quantities and their units and dimension vectors for the viscous pipe example.}
\label{tab:pipedim}
\begin{tabular}{llll}
physical quantity & symbol & units & dimension vector\\
\hline
fluid velocity & $V$ & $\text{m}\,\text{s}^{-1}$ & [0,1,-1]\\
fluid density & $\rho$ & $\text{kg}\,\text{m}^{-3}$ & [1,-3,0]\\
fluid viscosity & $\mu$ & $\text{kg}\,\text{m}^{-1}\,\text{s}^{-1}$ & [1,-1,-1]\\
pipe diameter & $D$ & $\text{m}$ & [0,1,0]\\
pipe roughness & $\varepsilon$ & $\text{m}$ & [0,1,0]\\
pressure loss & $dp/dx$ & $\text{kg}\,\text{m}^{-2}\,\text{s}^{-2}$ & [1,-2,-2]
\end{tabular}
\end{table}

To mimic the idealized experimental set up, we build a simple computational model. Given the independent variables, we solve \eqref{eq:colebrook} with a Newton method to estimate the dimensionless friction factor $\lambda=\lambda(V,\rho,\mu,D,\varepsilon)$; this model is valid over a wide range of flow conditions, i.e., laminar and turbulent regimes. Given the friction factor $\lambda$, we compute the pressure drop as~\cite[Chapter 5.2]{Palmer2008},
\begin{equation}
\label{eq:dpdx}
\frac{dp}{dx} \;=\; \frac{\lambda\,\rho\,V^2}{2\,D}.
\end{equation}
We consider three sets of flow conditions---loosely, laminar flow, turbulent flow, and high Reynolds number flow---each defined by its own probability density function on the independent variables. The range for each independent variable in the laminar case is in Table \ref{tab:bounds_laminar}; the ranges for the turbulent case are in Table \ref{tab:bounds_turbulent}; and the ranges for the high Reynolds case are in Table \ref{tab:bounds_highre}. Each set of ranges corresponds to a region of the Moody Diagram in Figure \ref{fig:moody}. Figure \ref{fig:moody_cont} shows the Moody Diagram as a contour plot, where the contours are the friction factor as a function of (the logs of) Reynolds number and roughness scale. Each outlined region corresponds to one of the flow cases of interest where we study the unique and relevant dimensionless groups. The red outline corresponds to laminar flow, the black outline corresponds to turbulent flow, and the magenta outline corresponds to the high Reynolds flow.

\begin{table}[h]
\centering
\caption{Bounds on the independent variables for the laminar flow case.}
\label{tab:bounds_laminar}
\begin{tabular}{lllll}
independent var. & symbol & lower bound & upper bound & units \\
\hline
fluid velocity & $V$ & $2.5\times10^{-2}$ & $3.0\times10^{-2}$ & $\unit{m}/\unit{s}$ \\
fluid density & $\rho$ & $1.0\times10^{-1}$ & $1.4\times10^{-1}$ & $\unit{kg}/\unit{m}^3$ \\
fluid viscosity & $\mu$ & $1.0\times10^{-6}$ & $1.0\times10^{-5}$ & $\unit{kg} / (\unit{m}\unit{s})$ \\
pipe diameter & $D$ & $5.0\times10^{-1}$ & $8.0\times10^{-1}$ & $\unit{m}$ \\
pipe roughness & $\varepsilon$ & $3.0\times10^{-5}$ & $8.0\times10^{-5}$ & $\unit{m}$
\end{tabular}
\end{table}

\begin{table}[h]
\centering
\caption{Bounds on the independent variables for the turbulent flow case.}
\label{tab:bounds_turbulent}
\begin{tabular}{lllll}
independent var. & symbol & lower bound & upper bound & units \\
\hline
fluid velocity & $V$ & $2.0\times10^{+0}$ & $4.0\times10^{+0}$ & $\unit{m}/\unit{s}$ \\
fluid density & $\rho$ & $1.0\times10^{-1}$ & $1.4\times10^{-1}$ & $\unit{kg}/\unit{m}^3$ \\
fluid viscosity & $\mu$ & $1.0\times10^{-6}$ & $1.0\times10^{-5}$ & $\unit{kg} / (\unit{m}\unit{s})$ \\
pipe diameter & $D$ & $5.0\times10^{-1}$ & $1.0\times10^{+0}$ & $\unit{m}$ \\
pipe roughness & $\varepsilon$ & $5.0\times10^{-4}$ & $2.0\times10^{-3}$ & $\unit{m}$ 
\end{tabular}
\end{table}

\begin{table}[h]
\centering
\caption{Bounds on the independent variables for the high Reynolds number case.}
\label{tab:bounds_highre}
\begin{tabular}{lllll}
independent var. & symbol & lower bound & upper bound & units \\
\hline
fluid velocity & $V$ & $5.0\times10^{+2}$ & $7.0\times10^{+2}$ & $\unit{m}/\unit{s}$ \\
fluid density & $\rho$ & $1.0\times10^{-1}$ & $1.4\times10^{-1}$ & $\unit{kg}/\unit{m}^3$ \\
fluid viscosity & $\mu$ & $1.0\times10^{-6}$ & $1.0\times10^{-5}$ & $\unit{kg} / (\unit{m}\unit{s})$ \\
pipe diameter & $D$ & $5.0\times10^{-1}$ & $1.0\times10^{+0}$ & $\unit{m}$ \\
pipe roughness & $\varepsilon$ & $1.0\times10^{-2}$ & $4.0\times10^{-2}$ & $\unit{m}$ 
\end{tabular}
\end{table}

\begin{figure}[ht]
\centering
\includegraphics[width=.8\textwidth]{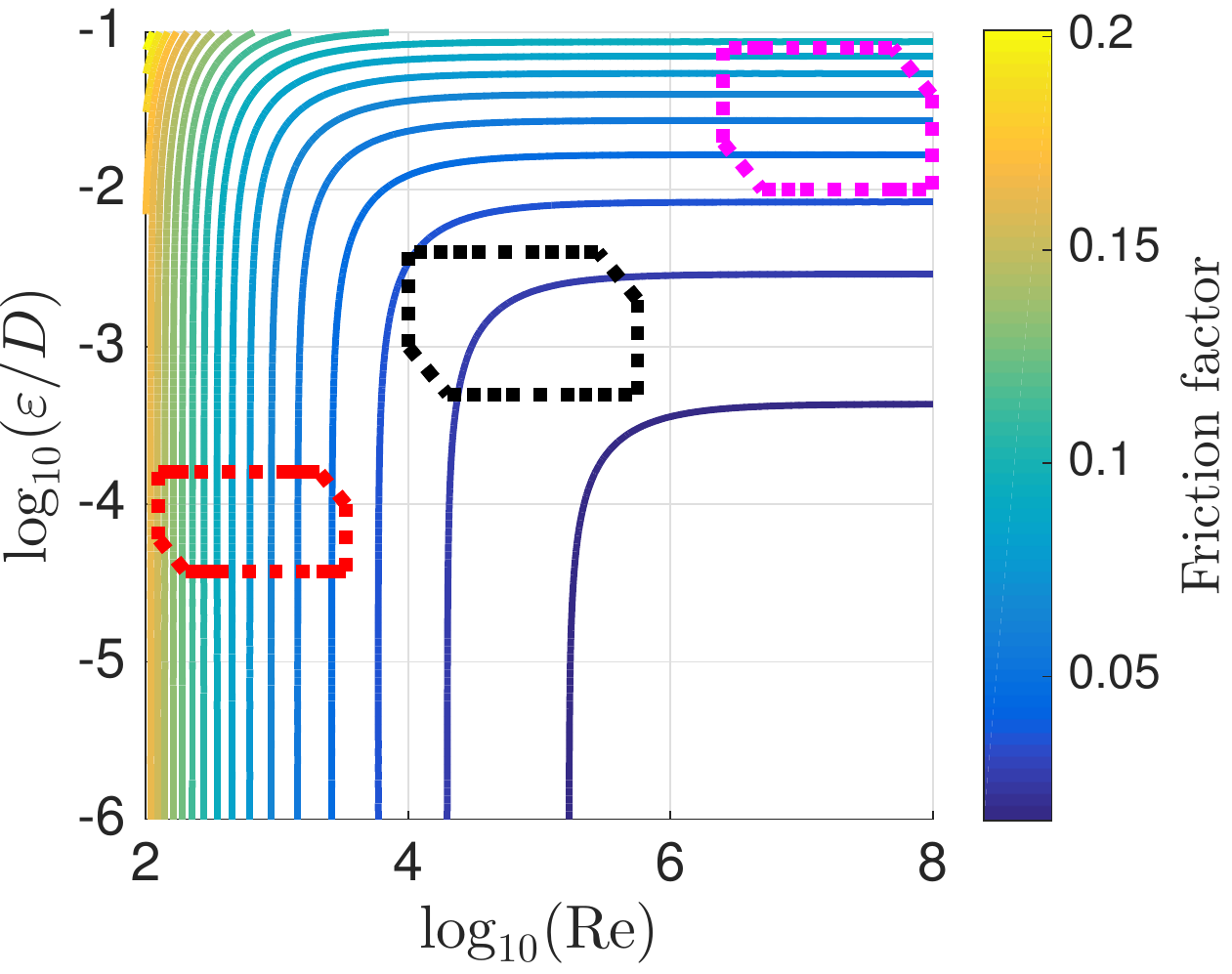}
\caption{The same data as the Moody Diagram in Figure \ref{fig:moody} in a contour plot; the contours represent the friction factor as a function of (i) the base ten log of Reynolds number and (ii) the base 10 log of the roughness scale. The three highlighted regions correspond to the flow regimes where we seek the unique and relevant nondimensional groups: the red outline is the laminar flow, the black outline is the turbulent flow, and the magenta outline is the high Reynolds flow.}
\label{fig:moody_cont}
\end{figure}

\subsection{Dimensional analysis}

\noindent
We first perform dimensional analysis to set up the numerical experiments. The matrix $\mD$ from \eqref{eq:D} contains the dimension vectors from Table \ref{tab:pipedim},
\begin{equation}
\mD \;=\;
\kbordermatrix{ & \rho & \mu & D & \varepsilon & V \cr
\text{kg} & 1 & 1 & 0 & 0 & 0\cr
\text{m} & -3 & -1 & 1 & 1 & 1\cr 
\text{s} & 0 & -1 & 0 & 0 & -1 }.
\end{equation}
The dimension vector for $dp/dx$ is $[1,-2,-2]^T$, and the vector $\vw$ that satisfies \eqref{eq:qoi}, i.e., that nondimensionalizes the pressure loss, is $[1,0,-1,0,2]^T$. The matrix $\mW$ whose columns span the 2-dimensional null space of $\mD$ are computed as the last two right singular vectors of $\mD$. Recall that the dimensional analysis subspace is the span of $\vw$ and $\mW$'s two columns; see \eqref{eq:A}.

\subsection{Pressure loss as a ridge function}

\noindent
With a numerical experiment, we can verify the derivation in Section \ref{ssec:deriv} that any empirical model is a ridge function. We choose $\rho$ as a uniform probability density on the hyperrectangle defined by the ranges in Tables \ref{tab:bounds_laminar} and \ref{tab:bounds_turbulent}. We use the script that computes $dp/dx$ given the independent variables to estimate the active subspaces; to estimate $\mC$ from \eqref{eq:Cmat}, we use a tensor product Gaussian quadrature rule for the approximate integration and first-order finite differences for the approximate partial derivatives. Lemma \ref{lem:Crank} says that the rank of $\mC$ should be at most 3, which is the number of independent variables (5) minus the number of base units (3) plus 1. We use a quadrature rule with 11 points per dimension---a total of 161051 points, which is sufficient for high accuracy. We then perform a finite difference convergence study as the finite difference step size goes to zero. 

\begin{figure}[ht]
\centering
\subfloat[Laminar]{
\label{fig:eiglam}
\includegraphics[width=.3\textwidth]{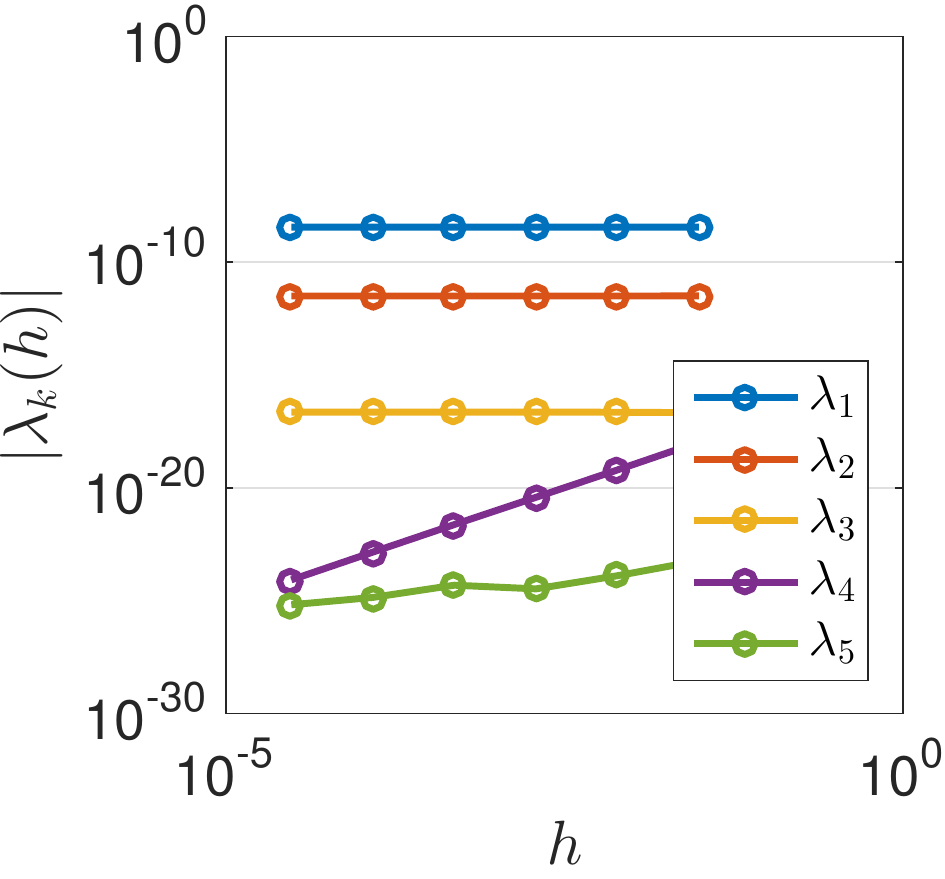}%
}
\subfloat[Turbulent]{
\label{fig:eigturb}
\includegraphics[width=.3\textwidth]{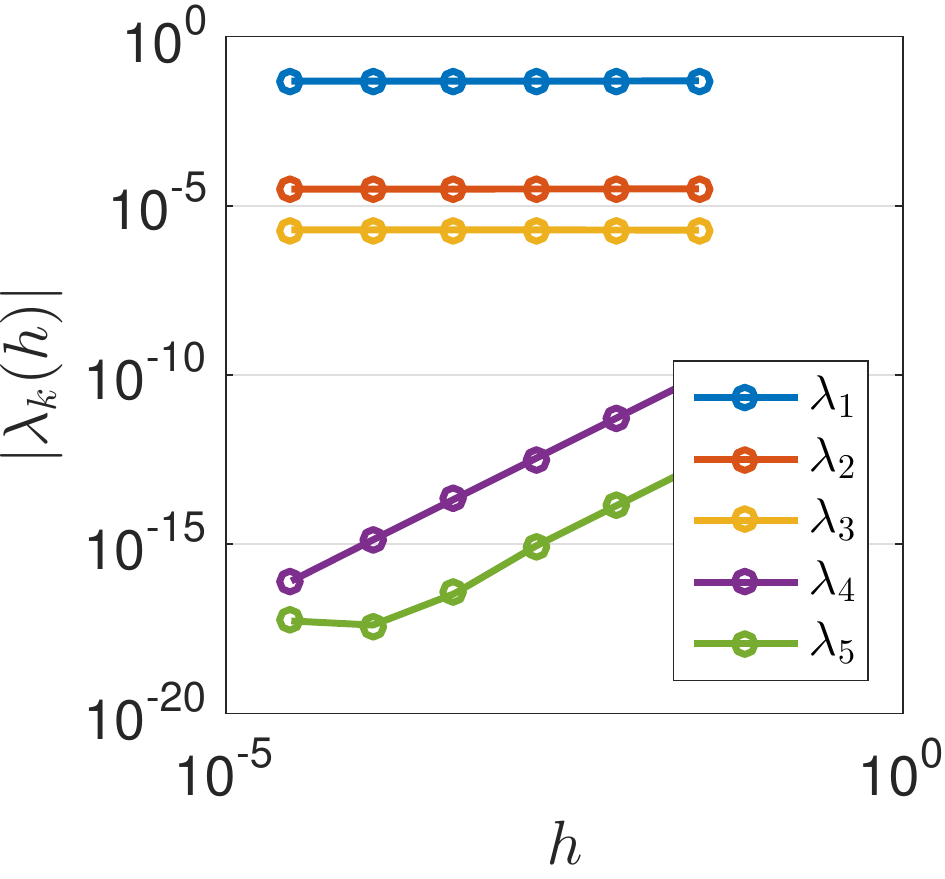}%
}
\subfloat[High Re]{
\label{fig:eighighre}
\includegraphics[width=.3\textwidth]{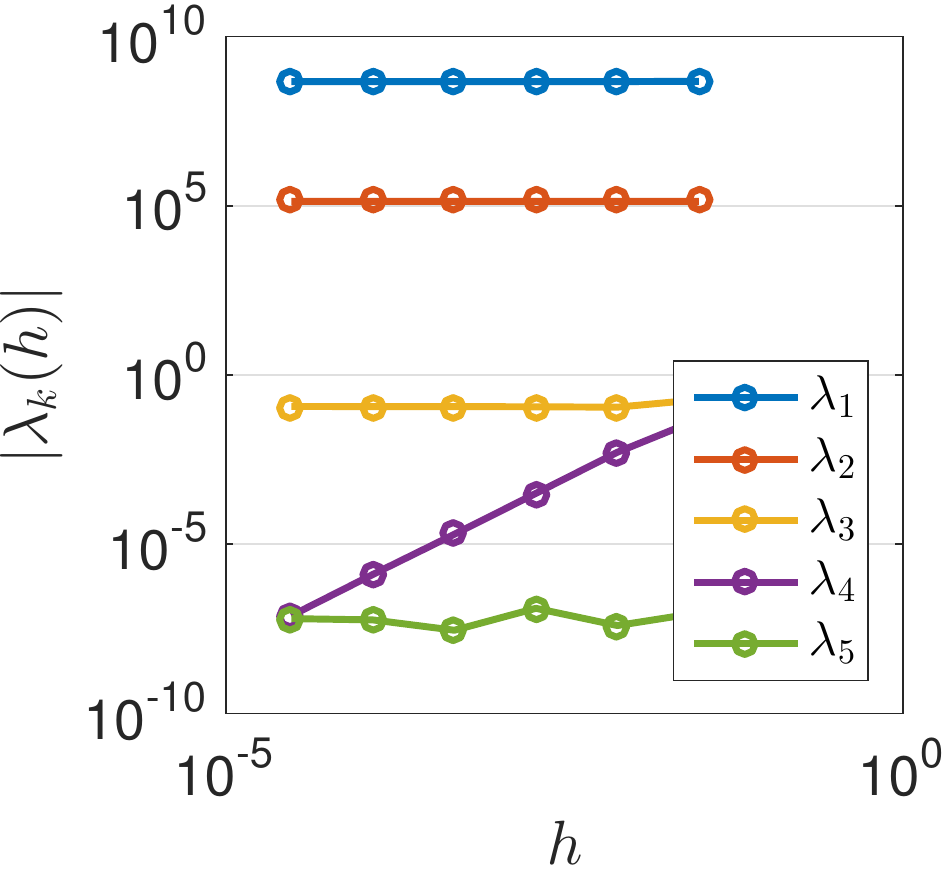}%
}
\caption{Convergence with respect to the finite difference step size $h$ of the eigenvalues from $\mC$ in \eqref{eq:Cmat} for the pressure loss $dp/dx$ as a function of the five independent physical variables. In all cases, the smallest two eigenvalues converge to zero as $h$ goes to zero, which is consistent with the three-dimensional ridge function structure derived in Section \ref{ssec:deriv}.}
\label{fig:eigfd}
\end{figure}

Figure \ref{fig:eigfd} shows the absolute values of the five eigenvalues of $\mC$ as the finite difference step size decreases; Figure \ref{fig:eiglam} shows the laminar flow case defined by the ranges in Table \ref{tab:bounds_laminar}; Figure \ref{fig:eigturb} shows the turbulent flow case defined by Table \ref{tab:bounds_turbulent}; and Figure \ref{fig:eighighre} shows the high Reynolds flow case defined by Table \ref{tab:bounds_highre}. In all cases, the two smallest eigenvalues go to zero at a first order rate as the step size $h$ goes to zero. This suggests that these two eigenvalues are nonzero for a finite step size because of finite difference approximation errors. The three largest eigenvalues converge to some nonzero values. This study verifies the fact, derived in Section \ref{ssec:deriv}, that $dp/dx$ as a function of the independent variables is a ridge function with active subspaces up to dimension 3---independent of the density function $\rho$.

\subsection{Unique and relevant dimensionless groups}

\noindent
We apply Algorithms \ref{alg:rs} and \ref{alg:fd} to the computational model that mimics the pipe experiments. For Algorithm \ref{alg:rs}, we use Matlab's Gaussian process regression (GPR) from the Statistics and Machine Learning Toolbox for the response surface. We use a full quadratic polynomial basis including cross terms, so there are six basis functions for the two dimensionless groups. We use the squared exponential kernel with independent coordinate length scales. The data for the GPR is 1000 pairs $\{(\pi^{(j)}, \bgamma^{(j)})\}$ from \eqref{eq:ghat}, where the $\{\bgamma^{(j)}\}$ are from Matlab's Latin hypercube design function \texttt{lhsdesign}. All parameters are tuned with maximum likelihood estimation using the built-in quasi-Newton solver. Oddly, Matlab's GPR implementation does not have have subroutines to compute gradients of the GPR prediction with respect to parameters, and the user cannot access the model terms needed to compute the GPR gradients directly; we implemented a first-order finite difference approximation for the GPR's gradients; preliminary experiments indicate that a step size of $h=1\times 10^{-6}$ produces 6-to-7 accurate digits. To estimate the eigenpairs in \eqref{eq:gas}, we use a Gauss-Legendre tensor product quadrature with 11 points per dimension, which produces 9-to-10 accurate digits. Note that we do not build the quadrature rule on the two-dimensional space of $\bgamma$ to estimate the integrals in \eqref{eq:gas}. Estimating the bivariate density function on the variables $\bgamma = \mW^T\log(\vq)$ would add another numerical approximation. Instead, we reformulate the integral in terms of $\vq$, where the density function $\rho$ from \eqref{eq:rho} is amenable to tensor product quadrature, i.e., the components of $\vq$ are independent. That is, we take advantage of the equality
\begin{equation}
\int \nabla \hat{g}(\bgamma)\, \nabla \hat{g}(\bgamma)^T\,\sigma(\bgamma)\,d\bgamma
\;=\;
\int \nabla \hat{g}(\mW^T\log(\vq))\, \nabla \hat{g}(\mW^T\log(\vq))^T\,\rho(\vq)\,d\vq,
\end{equation}
and estimate the integrals with quadrature on $\vq$. Admittedly, we have replaced the two-dimensional integrals by five-dimensional integrals for the sake of simpler quadrature implementation; the number of points used in the Gauss-Legendre quadrature is $11^5=161051$. If computing $\nabla g$ was particularly expensive---e.g., a complex simulation model or a physical experiment---then the dimension reduction of $\vq$ to $\bgamma$ to estimate integrals would warrant numerically estimating a quadrature rule for $\sigma(\bgamma)$.

For Algorithm \ref{alg:fd}, we use the same tensor product Gauss-Legendre quadrature rule with 161051 points in five dimensions (11 points per dimension). We ran a preliminary experiment to choose the finite difference step size. Figure \ref{fig:zerr} shows the relative convergence in the components of $\vz_1$ and $\vz_2$ from \eqref{eq:zhatfd} as $h$ decreases (right to left on the horizontal axis) for both the laminar and turbulent flow cases. The apparent first order convergence verifies the implementation. The minimum error suggests that $h=1\times 10^{-6}$ is an appropriate finite difference step size for 6-to-7 digits of accuracy. 

\begin{figure}[ht]
\centering
\subfloat[Laminar]{
\label{fig:zerrlam}
\includegraphics[width=.3\textwidth]{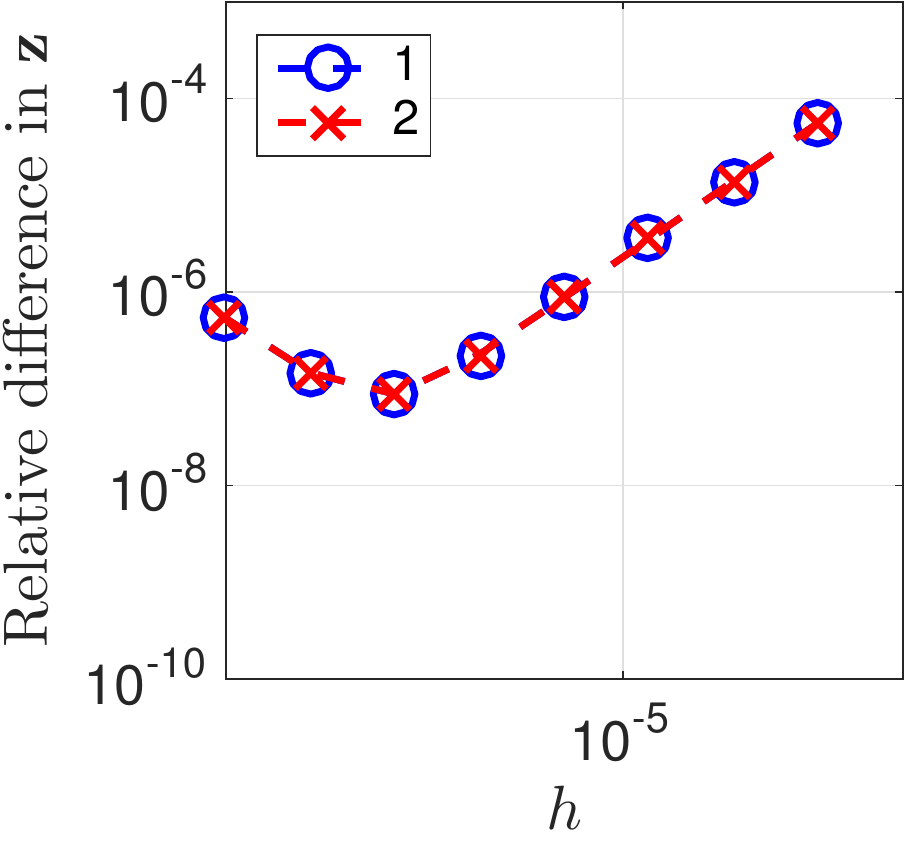}%
}
\subfloat[Turbulent]{
\label{fig:zerrturb}
\includegraphics[width=.3\textwidth]{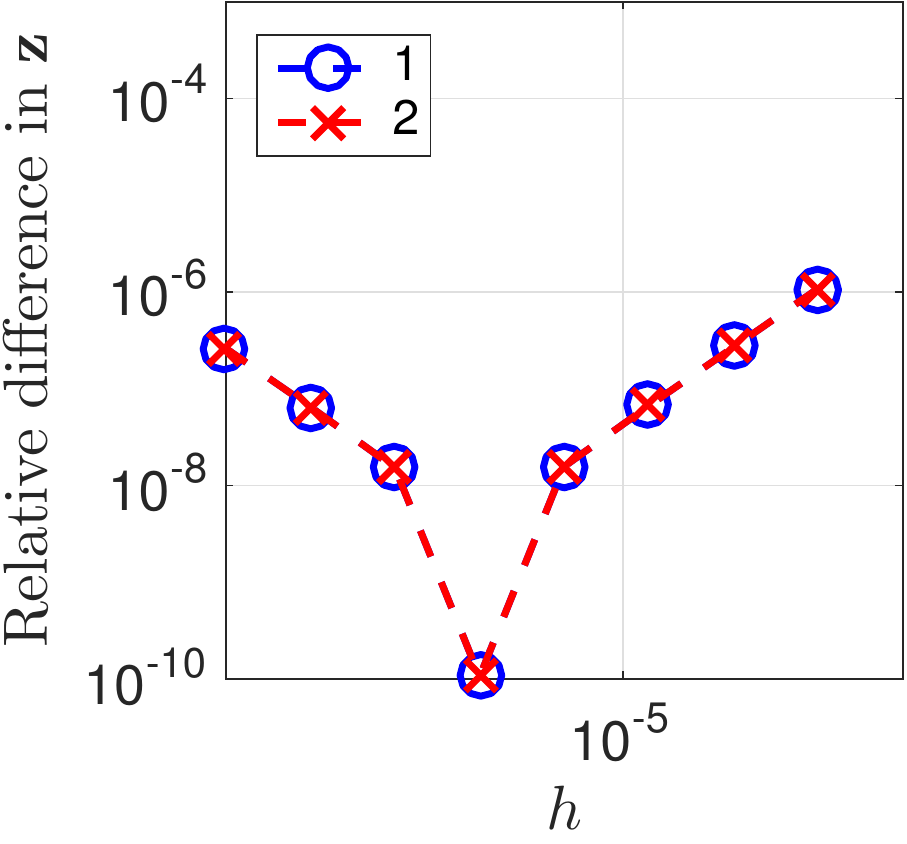}%
}
\subfloat[High Re]{
\label{fig:zerrhighre}
\includegraphics[width=.3\textwidth]{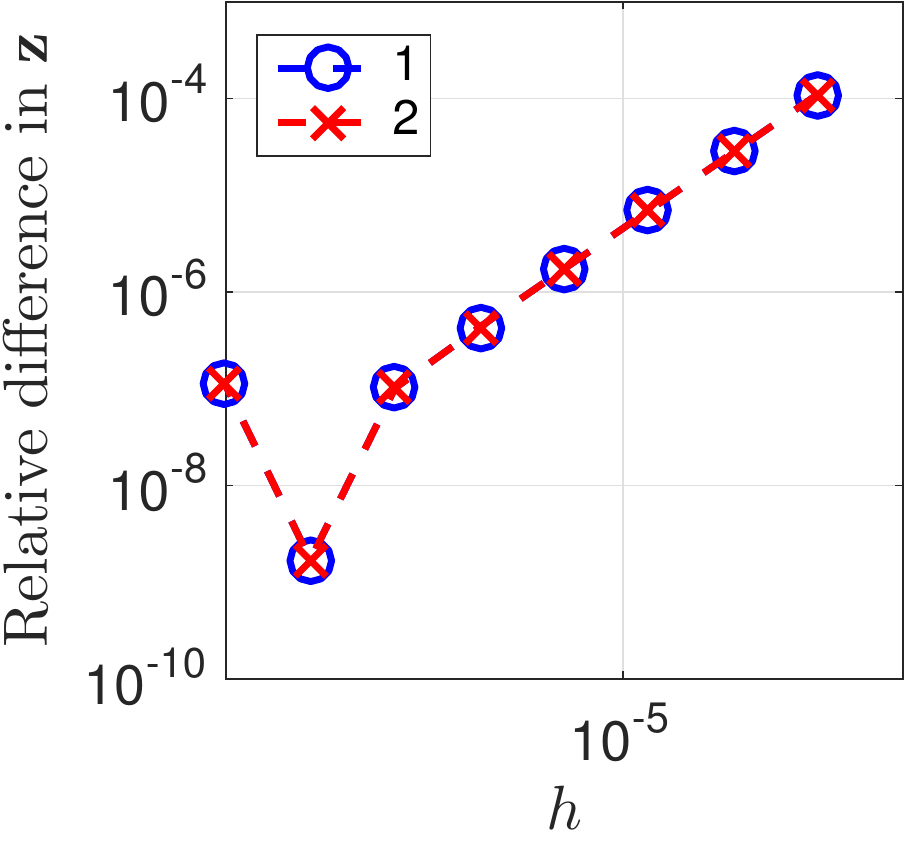}%
}
\caption{Convergence of the computed exponents $\vz_1$ and $\vz_2$ from Algorithm \ref{alg:fd} as the finite difference step size $h$ decreases. Figure \ref{fig:zerrlam} shows results for the laminar flow case; Figure \ref{fig:zerrturb} shows the turbulent flow case; and Figure \ref{fig:zerrhighre} shows the high Reynolds flow case.}
\label{fig:zerr}
\end{figure}

Table \ref{tab:laminar_z} shows the computed values of $\vz_1$ and $\vz_2$ with Algorithms \ref{alg:rs} and \ref{alg:fd} for the laminar flow case defined by the ranges in Table \ref{tab:bounds_laminar}. Recall that $\vz_1$ contains the exponents of the unique, most relevant dimensionless group, and $\vz_2$ contains the exponents of the second most relevant. For reference, we have included the exponents from classical dimensional analysis that correspond to the Reynolds number $\Ren$ and the dimensionless roughness scale $\varepsilon/D$. Up to scaling by 2, both algorithms find that the Reynolds number is the most relevant dimensionless group. The smaller eigenvalue associated with $\vz_2$ suggests that there is only one important dimensionless group (that is, the Reynolds number), and the components of $\vz_2$ are essentially random and meaningless. This is consistent with the Poiseuille relation \eqref{eq:poiseuille}, i.e., that the friction factor is inversely proportional to the Reynolds number for laminar flow; see~\cite[Chapter 5.2]{Palmer2008}. 

\begin{table}[h]
\centering
\caption{Computed exponents for the laminar pipe flow defined by the ranges in Table \ref{tab:bounds_laminar}}
\label{tab:laminar_z}
\begin{tabular}{l|ll|ll|ll|}
 & \multicolumn{2}{|l|}{DA} & \multicolumn{2}{|l|}{Alg \ref{alg:rs}} & \multicolumn{2}{|l|}{Alg \ref{alg:fd}} \\
variable & $\Ren$ & $\varepsilon/D$ & $\vz_1$ & $\vz_2$ & $\vz_1$ & $\vz_2$  \\
\hline
$\rho$ & 1.000 & 0.000 & 0.500 & 0.189 & 0.500 & -0.189 \\ 
$\mu$ & -1.000 & 0.000 & -0.500 & -0.189 & -0.500 & 0.189 \\
$D$ & 1.000 & -1.000 & 0.500 & -0.567 & 0.500 & 0.567 \\
$\varepsilon$ & 0.000 & 1.000 & 0.000 & 0.756 & -0.001 & -0.756 \\
$V$ & 1.000 & 0.000 & 0.500 & 0.189 & 0.500 & -0.189 \\
\hline
Eigenvalue & --- & --- & 2.19e-2 & 2.65e-8 & 2.39e-2 & 3.39e-9  
\end{tabular}
\end{table}

Table \ref{tab:turbulent_z} shows the computed values of $\vz_1$ and $\vz_2$ for the turbulent flow case defined by the ranges in Table \ref{tab:bounds_turbulent}; for reference we have computed the exponents from classical dimensional analysis. In this case, the two eigenvalues from each algorithm are similar orders of magnitude, which indicates that both dimensionless groups are important. The dimensionless groups identified by the algorithms are not the typical Reynolds number and roughness length scale from classical treatments of this viscous pipe flow problem~\cite[Chapter 5.2]{Palmer2008}. However, the dimensionless groups computed from Algorithms \ref{alg:rs} and \ref{alg:fd} can be expressed as products of powers of Reynolds number and roughness length scale---which is equivalent to expressing $\vz_1$'s and $\vz_2$'s from Table \ref{tab:turbulent_z} as linear combinations of the exponents defining Reynolds number and roughness length scale, i.e., the two leftmost columns in Table \ref{tab:turbulent_z}. Using the notation from Section \ref{ssec:unique}, the computed dimensionless groups from Algorithm \ref{alg:rs} (columns 3 and 4 in Table \ref{tab:turbulent_z}) can be expressed as
\begin{equation}
\hat{\pi}_1 = \pi_1^{0.304}\,\pi_2^{-0.429},\qquad
\hat{\pi}_2 = \pi_1^{0.440}\,\pi_2^{0.622},
\end{equation}
where $\pi_1$ is the Reynolds number and $\pi_2$ is the roughness scale---see columns 1 and 2 of Table \ref{tab:turbulent_z}, respectively. The exponents include three significant digits. The comparable exponents for $\hat{\pi}_1$ and $\hat{\pi}_2$ from Algorithm \ref{alg:fd} (columns 5 and 6 of Table \ref{tab:turbulent_z}) are 
\begin{equation}
\hat{\pi}_1 = \pi_1^{0.309}\,\pi_2^{-0.423},\qquad
\hat{\pi}_2 = \pi_1^{0.436}\,\pi_2^{0.627}.
\end{equation}
Recall that $\vz_1$ and $\vz_2$ depend on the chosen probability density $\rho$; the results from Table \ref{tab:turbulent_z} are not valid for all viscous pipe flows---only the flows consistent with the regime defined by the ranges in Table \ref{tab:bounds_turbulent}. 
Therefore, for flows consistent with the ranges in Table \ref{tab:bounds_turbulent}, the most relevant dimensionless group is not a common dimensionless group. It is a different product of powers of the original independent variables $\rho$, $\mu$, $D$, $\varepsilon$, and $V$. And the second most relevant dimensionless group is some other product of powers. 

\begin{table}[h]
\centering
\caption{Computed exponents for the turbulent pipe flow defined by the ranges in Table \ref{tab:bounds_turbulent}}
\label{tab:turbulent_z}
\begin{tabular}{l|ll|ll|ll|}
 & \multicolumn{2}{|l|}{DA} & \multicolumn{2}{|l|}{Alg \ref{alg:rs}} & \multicolumn{2}{|l|}{Alg \ref{alg:fd}} \\
variable & $\Ren$ & $\varepsilon/D$ & $\vz_1$ & $\vz_2$ & $\vz_1$ & $\vz_2$  \\
\hline
$\rho$        &  1.000 &  0.000 &  0.304 &  0.440 &  0.309 &  0.436 \\ 
$\mu$         & -1.000 &  0.000 & -0.304 & -0.440 & -0.309 & -0.436 \\
$D$           &  1.000 & -1.000 &  0.734 & -0.183 &  0.732 & -0.190 \\
$\varepsilon$ &  0.000 &  1.000 & -0.429 &  0.622 & -0.423 &  0.627 \\
$V$           &  1.000 &  0.000 &  0.304 &  0.440 &  0.309 &  0.436 \\
\hline
Eigenvalue & --- & --- & 3.34e-4 & 2.09e-5 & 3.58e-4 & 1.70e-5  
\end{tabular}
\end{table}

Table \ref{tab:highre_z} shows $\vz_1$ and $\vz_2$ computed with Algorithms \ref{alg:rs} and \ref{alg:fd} for the high Reynolds flow case defined by the ranges in Table \ref{tab:bounds_highre}. The exponents of the classical dimensionless groups for pipe flow are in the first two columns for reference. For both algorithms, the first eigenvalue is several orders of magnitude larger than the second, which suggests that the first active variable is much more important than the second. Moreover, the components of $\vz_1$---for both algorithms---are a scalar multiple of the exponents for the roughness scale. In other words, the algorithms identify the roughness scale as the only important parameter in this flow regime. This is consistent with the fluid dynamics theory that, for high Reynolds number, the pressure loss depends only on the roughness scale and not on the Reynolds number. 

\begin{table}[h]
\centering
\caption{Computed exponents for the high Reynolds pipe flow defined by the ranges in Table \ref{tab:bounds_highre}}
\label{tab:highre_z}
\begin{tabular}{l|ll|ll|ll|}
 & \multicolumn{2}{|l|}{DA} & \multicolumn{2}{|l|}{Alg \ref{alg:rs}} & \multicolumn{2}{|l|}{Alg \ref{alg:fd}} \\
variable & $\Ren$ & $\varepsilon/D$ & $\vz_1$ & $\vz_2$ & $\vz_1$ & $\vz_2$  \\
\hline
$\rho$        &  1.000 &  0.000 &  0.000 & -0.535 &  0.000 & -0.535 \\ 
$\mu$         & -1.000 &  0.000 &  0.000 &  0.535 &  0.000 &  0.535 \\
$D$           &  1.000 & -1.000 &  0.707 & -0.267 &  0.707 & -0.267 \\
$\varepsilon$ &  0.000 &  1.000 & -0.707 & -0.267 & -0.707 & -0.267 \\
$V$           &  1.000 &  0.000 &  0.000 & -0.535 &  0.000 & -0.535 \\
\hline
Eigenvalue & --- & --- & 5.30e-3 & 1.31e-10 & 5.71e-3 & 6.97e-11  
\end{tabular}
\end{table}

The geometric interpretation of the eigenvectors in \eqref{eq:gas} and \eqref{eq:evecsfd} enables useful visualizations---closely related to \emph{summary plots} from regression graphics~\cite{cook2009regression}---for understanding this differences between the computed $\vz_1$, $\vz_2$ and the classical dimensional analysis exponents. Recall that $\vz_i$ defines the unique and relevant dimensionless group $\hat{\pi}_i$ in \eqref{eq:udim}. Figure \ref{fig:ssp} contains scatter plots of the dimensionless output $\log(\pi)$ from \eqref{eq:qoi} (the color axis) versus two sets of dimensionless groups: $(\log(\pi_1), \log(\pi_2))$ in Figure \ref{fig:ssp_da} and $(\log(\hat{\pi}_1), \log(\hat{\pi}_2))$ in Figure \ref{fig:ssp_alg2}. The first set in Figure \ref{fig:ssp_da} are from classical dimensional analysis; $\pi_1$ is the Reynolds number and $\pi_2$ is the dimensionless roughness. The dimensionless groups $(\log(\hat{\pi}_1), \log(\hat{\pi}_2))$ in Figure \ref{fig:ssp_alg2} represent a rotation---approximately 127 degrees---from the classical dimensionless groups; the value ``127 degrees'' comes from the arcosine of the (1,1) element of $\hmU$ from \eqref{eq:evecsfd} interpreted as a two-dimensional rotation matrix. This rotation aligns the important directions with the coordinate axes, where \emph{importance} is precisely defined by the sensitivity analysis in Section \ref{ssec:sensitivity}. 

The plots in Figure \ref{fig:ssp_da} can be interpreted as zooming into the black outlined section of the Moody Diagram contour plot in Figure \ref{fig:moody_cont}. Recall that the black outline is a result of the ranges in Table \ref{tab:bounds_turbulent}. Figure \ref{fig:ssp_alg2} rotates the zoomed-in surface so that its important directions are coordinate-aligned. The contours within each of the outlined regions in Figure \ref{fig:moody_cont} are consistent with the computed dimensionless groups. The red outlined region corresponds to Table \ref{tab:laminar_z}, where the algorithms identify Reynolds number as the only important dimensionless group. And the black outlined region corresponds to Table \ref{tab:highre_z}, where the algorithms identify the roughness scale as the only important dimensionless group. The helpful visualization is not possible when there are more than two dimensionless groups. However, the insight rich interpretation is still valid; the unique and relevant dimensionless groups represent rotating the domain of the dimensionless dependent variable as a function of the dimensionless groups such that the resulting coordinates are ordered by importance. 

\begin{figure}[ht]
\centering
\subfloat[Turbulent, DA]{
\label{fig:ssp_da}
\includegraphics[width=.43\textwidth]{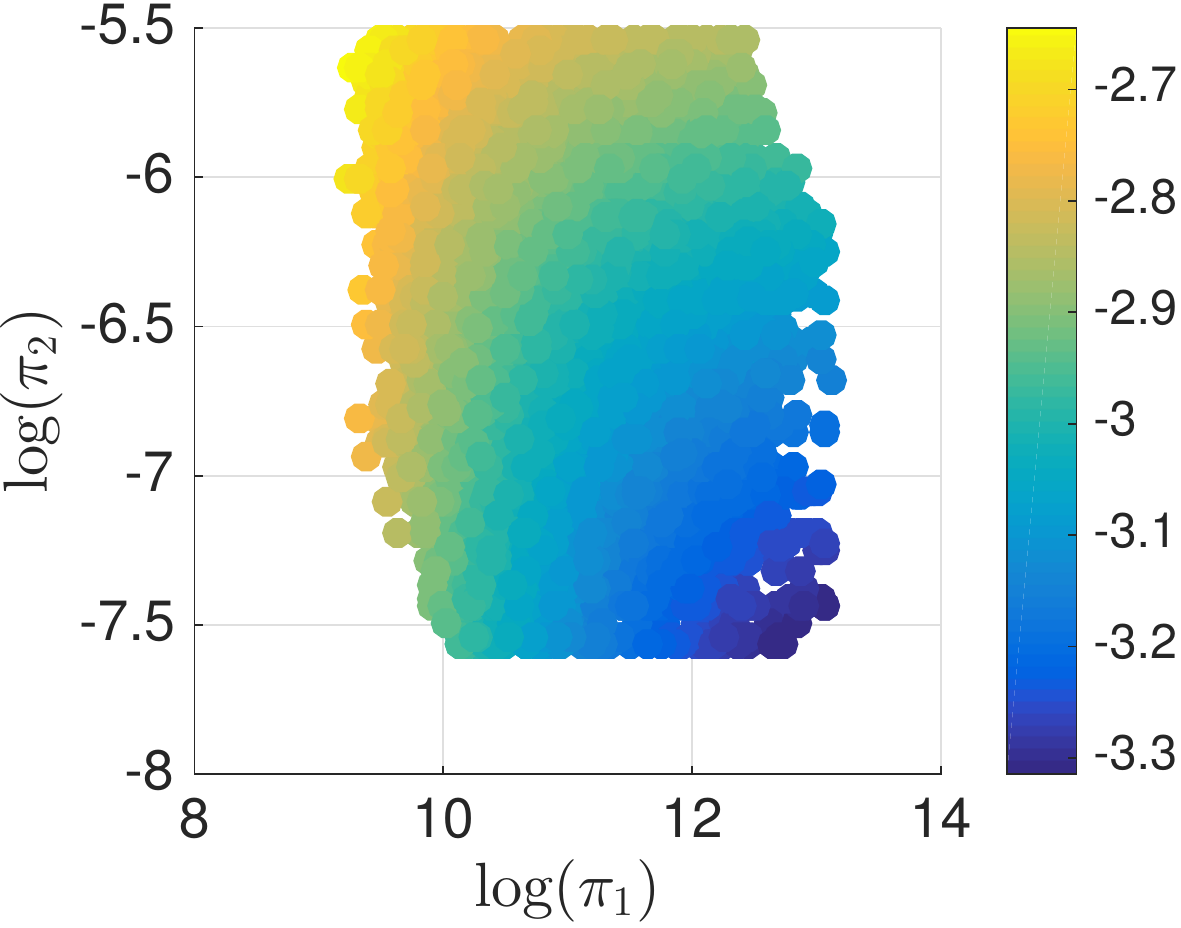}%
}
\subfloat[Turbulent, Alg \ref{alg:fd}]{
\label{fig:ssp_alg2}
\includegraphics[width=.43\textwidth]{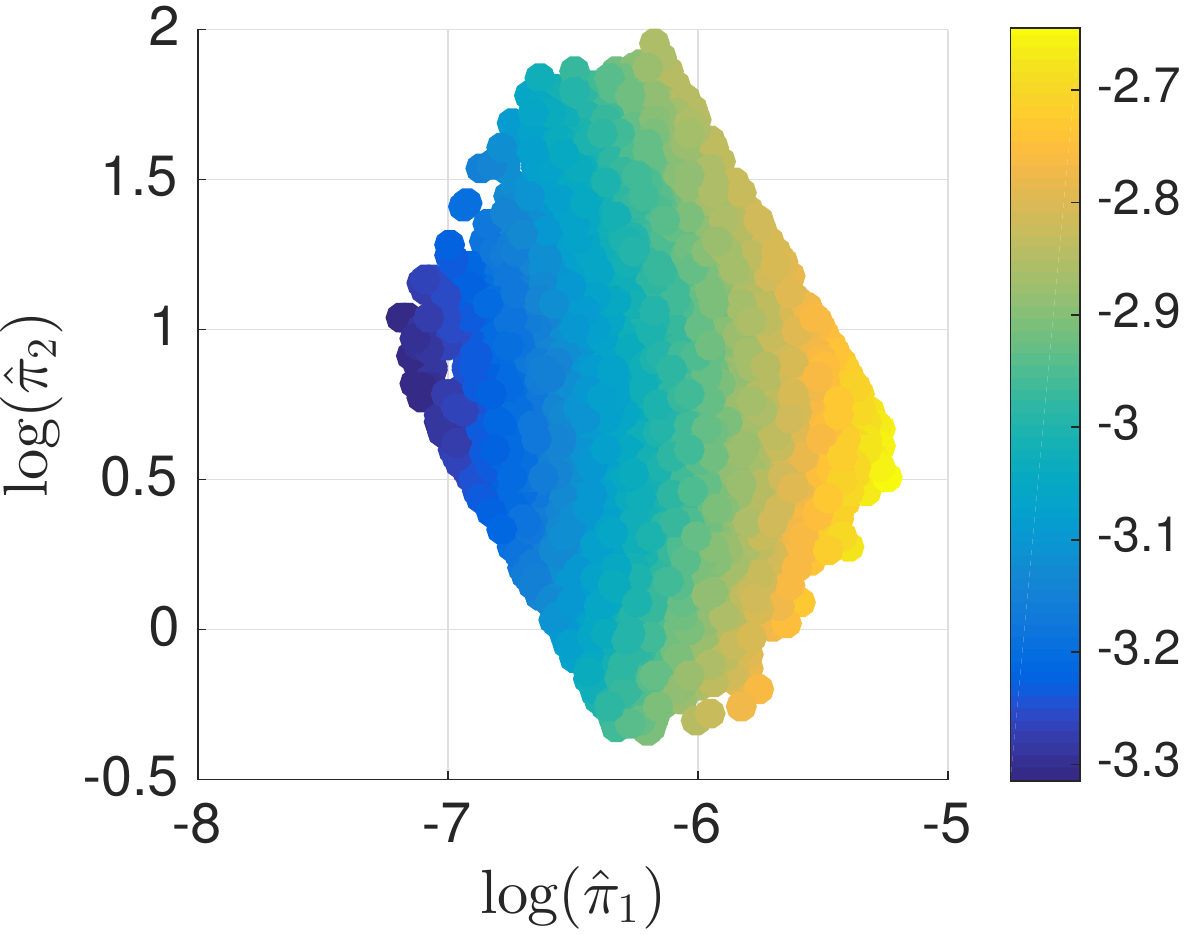}%
}
\caption{Scatter plots of the dimensionless dependent variable $\log(\pi)$ from \eqref{eq:qoi}, indicated by the color axis, versus two dimensionless groups---both for the turbulent flow case defined by Table \ref{tab:bounds_turbulent}. Figure \ref{fig:ssp_da} uses the classical dimensionless groups for the viscous pipe flow; $\pi_1$ is the Reynolds number and $\pi_2$ is the relative roughness. Figure \ref{fig:ssp_alg2} uses the dimensionless groups defined by the exponents $\vz_1$ and $\vz_2$ computed with Algorithm \ref{alg:fd} in Table \ref{tab:turbulent_z}. Figure \ref{fig:ssp_alg2} is roughly a 127-degree rotation of Figure \ref{fig:ssp_da}.}
\label{fig:ssp}
\end{figure}

We can compare cost and accuracy of the two algorithms from the numerical experiments in this section, which gives insight into whether one should use Algorithm \ref{alg:rs} or Algorithm \ref{alg:fd}. The results of Algorithm \ref{alg:rs} in Tables \ref{tab:laminar_z}, \ref{tab:turbulent_z}, and \ref{tab:highre_z} use 1000 runs of the computational model that mimics a pipe experiment to construct a response surface; in essence, the remainder of Algorithm \ref{alg:rs}, following response surface construction, is applying Algorithm \ref{alg:fd} to the constructed response surface. Although no more experiments (i.e., evaluations of the computer model) are needed, the computed exponents ($\vz_i$ from \eqref{eq:zhatrs}) depend on the response surface quality. In our numerical experiment, we did not study the effect of increasing the number of experiments used to fit the response surface; the convergence of the computed exponents with more training data will follow the convergence of the response surface---assuming the integration rule is sufficiently accurate for \eqref{eq:gas}. 

The results of Algorithm \ref{alg:fd} in Tables \ref{tab:laminar_z}, \ref{tab:turbulent_z}, and \ref{tab:highre_z} use 483153 runs of the computational pipe flow model. These runs are used to estimate the partial derivatives at each quadrature point; see \eqref{eq:evecsfd}. The increased cost produces greater accuracy since there is no error due to the response surface. 

Broadly, Algorithm \ref{alg:rs} is appropriate with a given set of experiments or when one can only afford a handful of experiments. In this case, the experimenter should beware that the accuracy in the computed exponents depends on the accuracy of the response surface. However, if experiments are relatively cheap---as in high-throughput experiments---then Algorithm \ref{alg:fd} will yield higher accuracy in the computed exponents.

\section{Summary and conclusions}
\label{sec:conclusions}

\noindent
Classical dimensional analysis (i) identifies dimensionless groups that constitute scale-free physical relationships and (ii) enables small-scale experiments that inform large-scale behavior. However, the dimensionless groups are not unique; nor do classical techniques provide any measure of importance or relevance of the dimensionless groups to the system. By supplementing the dimensional analysis with \emph{data} from an idealized experimental or computational setup, and by incorporating novel dimension reduction techniques related to global sensitivity analysis, we derived algorithms for constructing unique dimensionless groups equipped with measures of relative importance. The connections between dimensional analysis and active subspaces also show that any empirical model admits insight-rich low-dimensional structure in the map from independent variables to dependent variable. The two proposed algorithms are appropriate in different settings: the response surface-based algorithm is appropriate when only a small set of experiments are available and a global response surface is justified, while the finite difference-based algorithm is best suited for high-throughput experiments with fine control over independent variables. Both algorithms are well suited for computer experiments. We demonstrate both algorithms on a classical viscous pipe flow experiment, where a computer model based on a semi-empirical model mimics an experimental setup. 

Future work will account for noise in the experiments by including smoothing regularization in the response surfaces and derivative approximations. We expect that any system that benefits from dimensional analysis will also benefit from the proposed algorithms to identify unique and relevant dimensionless groups.

\section*{Acknowledgments}

\noindent
This material is based on work supported by Department of Defense, Defense Advanced Research Project Agency's program Enabling Quantification of Uncertainty in Physical Systems. The first author's work is partially supported by by the U.S. Department of Energy Office of Science, Office of Advanced Scientific Computing Research, Applied Mathematics program under Award Number DE-SC-0011077. The second author's work is supported by the National Science Foundation Graduate Research Fellowship under Grant No. DGE-114747.

\section*{References}

\bibliography{pi-subspaces}

\end{document}